\documentclass[11pt,letterpaper]{amsart}
\usepackage{
  amssymb, xcolor, mathtools, mleftright, xurl}
\usepackage[shortlabels]{enumitem}
\usepackage{hyperref}

\newcommand{\red}[1]{{#1}}

\newtheorem{theorem}{Theorem}[section]
\newtheorem{lemma}[theorem]{Lemma}

\newtheorem{prop}[theorem]{Proposition}
\newtheorem{corollary}[theorem]{Corollary}

\theoremstyle{definition}

\newtheorem{defn}[theorem]{Definition}

\theoremstyle{remark}
\newtheorem{rem}[theorem]{Remark}

\DeclareMathOperator{\ct}{ct}
\DeclareMathOperator{\ch}{char}
\DeclareMathOperator{\disc}{disc}
\DeclareMathOperator{\Ht}{Ht}
\DeclareMathOperator{\Htp}{Htp}
\DeclareMathOperator{\ind}{ind}
\DeclareMathOperator{\sgn}{sgn}
\newcommand{\monic}{\mathrm{monic}}
\DeclareMathOperator{\Res}{Res}
\DeclareMathOperator{\Hom}{Hom}
\DeclareMathOperator{\Disc}{Disc}
\DeclareMathOperator{\Ind}{Ind}
\DeclareMathOperator{\Frac}{Frac}
\DeclareMathOperator{\Aut}{Aut}

\newcommand{\GL}{\mathrm{GL}}
\newcommand{\CC}{\mathbb{C}}
\newcommand{\FF}{\mathbb{F}}
\newcommand{\QQ}{\mathbb{Q}}
\newcommand{\PP}{\mathbb{P}}
\newcommand{\RR}{\mathbb{R}}
\newcommand{\ZZ}{\mathbb{Z}}

\renewcommand{\aa}{\mathfrak{a}}
\newcommand{\pp}{\mathfrak{p}}

\newcommand{\E}{\mathcal{E}}
\newcommand{\I}{\mathrm{I}}
\newcommand{\OO}{\mathcal{O}}

\newcommand{\ba}{\overline}
\newcommand{\cross}{\times}

\newcommand{\textand}{\quad \text{and} \quad}
\renewcommand{\to}{\mathop{\rightarrow}\limits}
\newcommand{\size}[1]{\lvert #1 \rvert}
\let\left\mleft
\let\right\mright

\newcommand{\ceil}[1]{\left\lceil #1 \right\rceil}
\newcommand{\1}{\mathbf{1}}
\newcommand{\0}{\mathbf{0}}
\newcommand{\intsec}{\cap}
\newcommand{\bigintsec}{\bigcap}
\newcommand{\isom}{\cong}
\newcommand{\<}{\left\langle}
\renewcommand{\>}{\right\rangle}
\renewcommand{\(}{\left(}
\renewcommand{\)}{\right)}
\newcommand{\ds}{\displaystyle}

\renewcommand{\epsilon}{\varepsilon}

\title[Galois groups of reciprocal polynomials]{Galois groups of reciprocal polynomials\\and the van der Waerden--Bhargava theorem}
\author{Theresa C. Anderson, Adam Bertelli, and Evan M. O'Dorney}

\begin{document}
  \begin{abstract}
    We study the Galois groups $G_f$ of degree $2n$ reciprocal (a.k.a.\ palindromic) polynomials $f$ of height at most $H$, finding that $G_f$ falls short of the maximal possible group $S_2 \wr S_n$ for a proportion of all $f$ bounded above and below by constant multiples of $H^{-1} \log H$, both in the nonmonic setting and, for $n \geq 3$, in the monic setting. This answers a 1998 question of Davis--Duke--Sun and extends Bhargava's 2023 resolution of van der Waerden's 1936 conjecture on the corresponding question for general polynomials. Unlike in that setting, the dominant contribution comes not from reducible polynomials but from those $f$ for which $(-1)^n f(1) f(-1)$ is a square, causing $G_f$ to lie in an index-$2$ subgroup.
    
    {\textsc{AI statement.}} ChatGPT 5.6 Sol was used to help find and correct typos and mathematical errors in earlier drafts of this paper. All remaining errors are the responsibility of the authors.
  \end{abstract}
  
  
  \maketitle
  
  \section{Introduction}
  For a positive integer $n$, let $E_n(H)$ denote the number of degree $n$ monic separable polynomials $f(x) = x^n + a_{n-1} x^{n-1} + \cdots + a_1 x + a_0$ with integer coefficients $a_i \in [-H, H]$ whose Galois group is \emph{not} $S_n$. Hilbert's Irreducibility Theorem implies that Galois groups not equal to $S_n$ occur $0\%$ of the time, in other words,
  \[
  E_n(H) = o(H^n).
  \]
  In 1936, van der Waerden \cite{vdW1936} gave a quantitative upper bound and conjectured that the true order of growth is
  \begin{equation} \label{eq:vdW}
    E_n(H) \asymp H^{n-1},
  \end{equation}
  the lower bound coming from counting multiples of $x$ (or any fixed monic linear polynomial). For the next several decades, progress continued, with many authors making improvements on the bound, including Knobloch (1955) \cite{Knobloch1955}, Gallagher (1972) \cite{Gallagher1972}, 
  Chow and Dietmann (2020) \cite{CD2020} who proved it for $n \leq 4$, and a group including the first author (2023) \cite{6author}. The conjecture \eqref{eq:vdW} was finally resolved by Bhargava (2023) (\cite{Bhargava_vdW}; an abridged version of the paper has appeared in print \cite{Bhargava_vdW_short}). Bhargava's method harnesses sophisticated known results (classification of subgroups of $S_n$, distribution of discriminants of number fields) in combination with innovative recent methods, including Fourier equidistribution and the use of the \emph{double discriminant} $\disc_{a_n} \disc_x f$. \red{For transitive subgroups such as $A_n$, the frequency of appearance is expected to be much lower \cite{BSBPM24}.}
  
  In this paper, we study the analogous question for the subspace of reciprocal polynomials, which was previously posed by Davis--Duke--Sun \cite{DDS1998}. A polynomial $f$ is called \emph{reciprocal} if
  \[
  f(1/x) = \frac{1}{x^{\deg f}} f(x),
  \]
  in other words, if the coefficient list of $f$ is palindromic. We focus on the case where $f$ is of even degree $2n$ (in the odd-degree case, $f$ must be reducible, equal to $x+1$ times an even-degree reciprocal polynomial). The roots of $f$ come in reciprocal pairs $\{\alpha, 1/\alpha\}$; the Galois group must preserve the partition into pairs and thus is a subgroup of the wreath product $S_2 \wr S_n$, the subgroup of $S_{2n}$ of order $2^n n!$ preserving this partition. In this paper, we provide a precise estimate for how often the group is strictly smaller:
  \begin{theorem} \label{thm:main} Let $\E_n^\monic(H)$ be the number of separable monic reciprocal polynomials $f$ of degree $2n$ with coefficients in $[-H,H]$ whose Galois group is not $S_2 \wr S_n$. Then for each $n \geq 3$,
    \[
    \E_n^\monic(H) \asymp H^{n - 1} \log H.
    \]
  \end{theorem}
  \begin{rem}
    Here and throughout the paper, if $f(H), g(H)$ are real-valued functions of a sufficiently large real number $H$, then the usual notations
    \[
    f(H) \ll g(H), \quad g(H) \gg f(H), \quad f(H) = O\bigl(g(H)\bigr)
    \]
    mean that $\size{f(H)} < c \cdot \size{g(H)}$ for sufficiently large $H$ and some constant $c$, while the notations
    \[
    f(H) \asymp g(H), \quad f(H) = \Theta\bigl(g(H)\bigr)
    \]
    mean that $f(H) \ll g(H)$ and $f(H) \gg g(H)$. Finally, $f(H) = o\bigl(g(H)\bigr)$ means that $\lim_{H \to \infty} f(H)/g(H) = 0$. The implied constants may depend on $n$ but not on $H$.
  \end{rem}
  
  Theorem \ref{thm:main} is significant for several reasons:
  \begin{itemize}
    \item This is a sharp improvement on the work of Davis--Duke--Sun \cite{DDS1998}, who show $\E_n^\monic(H) \ll H^{n-1/2} \log H$.
    \item The correct order of growth is \emph{not} $H^{n-1}$, as one might expect by counting the reducible polynomials by analogy with van der Waerden's conjecture. Instead, the dominant contribution comes from reciprocal polynomials $f$ such that $(-1)^n f(1) f(-1)$ is a square. This makes $\disc f$ a square and causes the Galois group to lie in an index-$2$ subgroup, called $G_1$ in the classification below. Reciprocal $f$ such that $(-1)^n f(1) f(-1)$ is a square (among other conditions) show up as characteristic polynomials of automorphisms of lattices in \cite{GroMcM02}.
    \item Since the characteristic polynomial of a symplectic matrix is reciprocal, a potential application of this work is to understand the characteristic polynomials of random elements of $\mathrm{Sp}_{2n}(\ZZ)$, extending the work of Rivin \cite{Rivin08} for $\operatorname{SL}_n(\ZZ)$ and that of the first author and Lemke Oliver \cite{KALO} \red{for the symplectic case. Since the Weyl group of $\mathrm{Sp}_{2n}$ is $S_2 \wr S_n$, this fits into a wider principle that a random element of a connected split reductive group $G/\QQ$, faithfully embedded, should almost surely yield a number field whose Galois group is the Weyl group of $G$ \cite{JKZ13,LR14}.} Additionally, there are interesting connections to results in quantum chaos (see below).
  \end{itemize}
  
  \red{Throughout this paper, $n$ is fixed and the height $H \to \infty$ (the ``large box model''). Work has also been fruitful for the complementary question on the ``restricted coefficient model'' where the coefficients are drawn from a fixed set and $n \to \infty$ \cite{BV19,BSK20,BSKK23}. In particular, Hokken \cite{Hokken23} has recently proved an asymptotic count of the number of $\pm1$-coefficient reciprocal polynomials where the discriminant is a square, that is, $G_f \subseteq G_1$ in the notation below.}
  
  As with Bhargava \cite{Bhargava_vdW} and many other papers on van der Waerden's conjecture, the methods are largely indifferent to whether we look at monic polynomials or general non-monic polynomials $f$ with all coefficients ranging through a box $[-H,H]^{n+1}$. The analogue of Theorem \ref{thm:main} for the non-monic setting is:
  \begin{theorem} \label{thm:main_nonmonic} Let $\E_n(H)$ be the number of separable reciprocal polynomials $f$ of degree $2n$ with coefficients in $[-H,H]$ whose Galois group is not $S_2 \wr S_n$. Then for all $n \geq 1$,
    \[
    \E_n(H) \asymp H^{n} \log H.
    \]
  \end{theorem}
  \begin{rem}
    Here the range of applicability is $n \geq 1$. In the monic setting, the case $n=1$ is false because a monic reciprocal quadratic polynomial $f(x) = x^2 + a_1 x + 1$ has full Galois group $S_2$ for all $a_1 \neq \pm 2$, while the case $n=2$ remains open; see Remark \ref{rem:C4} below.
  \end{rem}
  
  Note that a degree $2n$ polynomial $f$ is reciprocal if and only if the Cayley-transformed polynomial
  \[
  \widetilde{f}(x) = (1+x)^{2n} f\(\frac{1-x}{1+x}\)
  \]
  is even, that is, $\widetilde{f}(x) = \widetilde{g}(x^2)$ for a polynomial $\widetilde g$. (This is a straightforward computation and well known, see for instance \cite[p.~275]{MP2017}.) For an even polynomial, the roots again come in pairs $\{\alpha, -\alpha\}$. So we also have the following corollary:
  \begin{corollary} \label{cor:even}
    The number of degree $2n$ even polynomials $\widetilde f$ with coefficients in $[-H,H]$ whose Galois group is not $S_2 \wr S_n$ is $\Theta(H^n \log H)$.
  \end{corollary}  
  Note that in this setting, the condition for the Galois group to lie in $G_1$ is that the product $(-1)^n a_{2n} a_0$ of the first and last coefficients of $\widetilde f$ be a square. If we impose $a_{2n} = 1$, the likelihood of this rises to $O(H^{-1/2})$, so the na\"ive analogue of Corollary \ref{cor:even} in the monic setting is false.
  
  \begin{rem} \label{rem:C4}
  The monic case $n=2$ remains open due to the difficulty of counting reciprocal quartic polynomials 
  \[
    f(x) = x^4 + a x^3 + b x^2 + a x + 1, \qquad |a|, |b| \leq H
  \]
  with cyclic Galois group $C_4$. Straightforward use of the large sieve shows that there are at most $O_\epsilon(H^{3/2 + \epsilon})$ of these, while $O_\epsilon(H^{1 + \epsilon})$ is needed to finish Theorem \ref{thm:main} for the missing case $n = 2$. From numerical experiments up to $H = 10^5$, it appears that the true count of such polynomials is closer to $H^{1/2 + \epsilon}$. A lower bound of $\Omega(H^{1/2})$ can be obtained by noting the family $a = 4$, $b = -10 - t^2$, $t \in \ZZ \setminus \{0, \pm 3\}$.
  \end{rem}
  
  As alluded to earlier, there are potential applications of our work to the area of quantum chaos.  In particular, studying the mass of eigenfunctions in quantum chaos is an area of great mathematical interest.  This area shares exciting connections with reciprocal polynomials via quantum cat maps, a toy model of study in the area, that are given by symplectic matrices $A$.  In particular, an important object to study the mass of eigenfunctions is the semiclassical measure $\mu$ associated to $A$.  It turns out that $\mu$ has nice support properties if the characteristic polynomial of $A^m$ for all $m \in \mathbb{N}$ (including $A$ itself) is irreducible over the integers.  In an appendix to a paper of Elena Kim, the first author and Lemke Oliver show that this irreducibility happens 100\% of the time and moreover that the generic Galois group of such polynomials is the wreath product $S_2 \wr S_n$ (see \cite{KALO} for more precise information and connections).
  
  Based on this recent result, we conjecture that if $M$ is a random $n \times n$ matrix, then the characteristic polynomials not only of $M$ but of its powers $M^i, i \geq 1$, are almost surely $S_n$, where ``almost surely'' entails error bounds of van der Waerden type. This can be interpreted as saying that, if $f$ is the characteristic polynomial, then not only its root $\alpha$ but each of its powers $\alpha^i$, $i \geq 1$, generates an $S_n$-extension of $\QQ$. The present work can be regarded as an extension of this to the fractional power $\sqrt{\alpha}$. It is natural to hope that the higher-order roots $\sqrt[m]{\alpha}$ of polynomials $g(x^m)$ can be handled in a similar way, the generic Galois group now being the semidirect product
  \[
  (\mu_m \wr S_n) \rtimes (\ZZ/m\ZZ)^\cross\red{,}
  \]
  \red{where $(\ZZ/m\ZZ)^\cross$ acts naturally on the group $\mu_m$ of $m$th roots of unity.}
  
  Our methods are parallel to Bhargava's in \cite{Bhargava_vdW}, with judicious modifications when necessary. In Section \ref{sec:recip_polys}, we lay out preliminary facts about reciprocal polynomials. In Section \ref{sec:groups}, we classify maximal subgroups of $S_2 \wr S_n$, and in particular, we show that we can narrow our focus to three groups, which we denote $G_1$, $G_2$, and $G_3$. In Sections \ref{sec:G1}, \ref{sec:G2}, and \ref{sec:G3}, we count the number $\E_n(G_i; H)$ of polynomials having each of those Galois groups (or a subgroup thereof). The $G_1$-polynomials we can count by direct parametrization. To count the $G_3$-polynomials $f$, we take advantage of the fact that $f$ is reducible over a quadratic extension of $\QQ$. The heart of the argument is to handle $G_2$, which plays a challenging role like that of the alternating group $A_n$ in the study of general polynomials. Following Bhargava, we divide up the polynomials into three cases based on the size of the discriminant and its prime divisors. While attacking each case in turn, we are led to apply Fourier analysis for equidistribution and to construct a suitably modified double discriminant. The Fourier analysis is more involved than Bhargava's and involves breaking up the desired count into terms supported on different sublattices of the lattice $V(\ZZ)$ of reciprocal polynomials.  This type of decomposition is used in harmonic analysis to split up a function into pieces with different Fourier properties, but our use is perhaps novel in this setting.
  
  Because Theorems \ref{thm:main} and \ref{thm:main_nonmonic} are so similar, we focus on Theorem \ref{thm:main_nonmonic}, the non-monic setting, which is the technically simpler of the two. At the end of each of Sections \ref{sec:G1}, \ref{sec:G2}, and \ref{sec:G3}, we explain how the proof must be adapted to the monic case to prove Theorem \ref{thm:main}.
  
  \subsection{Acknowledgements} TCA was partially supported by the National Science Foundation (grants 2231990, 2237937). We thank Hongyi (Brian) Hu and the rest of the participants in the Research Topics in Mathematical Sciences semester at CMU in spring 2024 for their contributions to the discussion. We thank Arno Fehm, Benedict Gross, Igor Shparlinski\red{, and the anonymous referees} for insightful comments. We thank Robert Lemke Oliver for a careful read of an earlier draft.
  
  \section{Reciprocal polynomials}\label{sec:recip_polys}
  We define the \emph{height} of an integer polynomial
  \[
  P(x) = c_n x^n + c_{n-1} x^{n-1} + \cdots + c_1 x + c_0 \in \ZZ[x]
  \]
  to be the maximum of the absolute values of the coefficients,
  \[
  \Ht P = \max\bigl\{\size{c_n}, \size{c_{n-1}}, \ldots, \size{c_0}\bigr\} \in \ZZ_{\geq 0}.
  \]
  Let
  \[
  f(x) = a_0 x^{2n} + a_1 x^{2n-1} + \cdots + a_{n-1} x^{n+1} + a_n x^n + a_{n-1} x^{n-1} + \cdots + a_1 x + a_0
  \]
  be a reciprocal polynomial of degree $2n$.
  Note that there is a unique degree $n$ integer-coefficient polynomial
  \[
  g(u) = b_n u^n + b_{n-1} u^{n-1} + \cdots + b_1 u + b_0
  \]
  such that
  \[
  f(x) = x^n g\(x + \frac{1}{x}\).
  \]
  The passage between $f$ and $g$ is bijective and linear, and it increases or decreases heights by at most a bounded factor, depending only on $n$. Hence it is immaterial whether we count $f$ or $g$ of height at most $H$. For most purposes, it is more convenient to count $g$.
  
  We denote the roots of $f$ by
  \[
  \alpha_1, \frac{1}{\alpha_1}, \alpha_2, \frac{1}{\alpha_2}, \ldots, \alpha_n, \frac{1}{\alpha_n}.
  \]
  Then the roots of $g$ are $\beta_1,\ldots, \beta_n$, where
  \[
  \beta_i = \alpha_i + \frac{1}{\alpha_i}.
  \]
  We will sometimes write $\alpha = \alpha_1$ and $\beta = \beta_1$ when the choice of root is irrelevant.
  
  In view of the main theorem we would like to prove, we can assume any statement that occurs for all but $O(H^{n})$ of the $O(H^{n + 1})$ polynomials $g$ of height at most $H$. For example, we may assume that $g$ is irreducible and that $g(2)$ and $g(-2)$ are nonzero. We have the tower of number fields
  \[
  K_f = \QQ(\alpha) \quad \supseteq \quad K_g = \QQ(\beta) \quad \supset \quad \QQ,
  \]
  where $K_g$ is an $S_n$-extension of degree $n$, while $K_f/K_g$ is of degree at most $2$, given by $K_f = K_g(\sqrt{\beta^2 - 4})$. Let $\widetilde{K}_g$ and $\widetilde{K}_f$, respectively, be the splitting fields of $g$ and $f$, and let $G_g$ and $G_f$ be their respective Galois groups, which are subgroups of $S_n$ and $S_2 \wr S_n$, with $G_f \twoheadrightarrow G_g$ under the natural projection $S_2 \wr S_n \twoheadrightarrow S_n$. By the main result of Bhargava \cite[Theorem 1]{Bhargava_vdW}, we can assume that $G_g$ is the whole $S_n$. Our aim in this paper is to understand when $G_f$ is not the whole $S_2 \wr S_n$.
  
  By the usual formula for the discriminant of a number field tower, we have
  \[
  (\Disc K_f) = (\Disc K_g)^2 \cdot N_{K_g/\QQ} \disc_{K_g} K_f
  \]
  as ideals in $\ZZ$. The following is a closely related result on the discriminants of the associated polynomials.
  \begin{lemma} \label{lem:disc_f}
    $\disc f = g(2) g(-2) (\disc g)^2$.
  \end{lemma}
  
  \begin{proof}
    We have
    \begin{align*}
      \disc f &= a_0^{4n - 2} \prod_{i=1}^n (\alpha_i-\alpha_i^{-1})^2 \times {} \\
      &\quad {} \times  \prod_{i<j}(\alpha_i-\alpha_j)^2(\alpha_i-\alpha_j^{-1})^2(\alpha_i^{-1}-\alpha_j)^2(\alpha_i^{-1}-\alpha_j^{-1})^2 \\
      &= b_n^{4n - 2} \prod_{i=1}^n (\beta_i^2-4)\cdot \prod_{i<j}\alpha_i^{-4}\alpha_j^{-4}(\alpha_i-\alpha_j)^4(\alpha_i\alpha_j-1)^4\\
      &= b_n \prod_{i=1}^n (2-\beta_i) \cdot b_n \prod_{i=1}^n (-2-\beta_i) \cdot b_n^{4n - 4} \prod_{i<j}(\beta_i-\beta_j)^4\\
      &= g(2)\cdot g(-2)\cdot(\disc g)^2. \qedhere
    \end{align*}
  \end{proof}

  \section{Maximal subgroups of \texorpdfstring{$S_2 \wr S_n$}{S₂ ≀ Sₙ}}\label{sec:groups}
  
  If $G_f$ is not the full $S_2 \wr S_n$, it is contained in a maximal subgroup. We have $S_2 \wr S_n = \FF_2^n \rtimes S_n$. The following elementary lemma classifies the maximal subgroups of a semidirect product whose normal factor is abelian.
  
  \begin{lemma}\label{lem:max_semidir}
    Let $X \rtimes S$ be a semidirect product of groups, with $X$ abelian. The maximal subgroups of $X \rtimes S$ are of two types:
    \begin{enumerate}[$($a$)$]
      \item\label{it:fibers} $X \rtimes S'$, for $S' < S$ a maximal subgroup;
      \item\label{it:Y} Those groups $G < X \rtimes S$ for which $Y = G \intsec X$ is a maximal $S$-invariant subgroup of $X$ such that the projection $G \to S$ is surjective.
    \end{enumerate}
  \end{lemma}
  \begin{proof}
    Let $G < X \rtimes S$ be a maximal subgroup, and let $S'$ be the projection of $G$ onto $S$. If $S' \neq S$, then $G \leq X \rtimes S'$, and we must have equality so we get a group of type \ref{it:fibers}.
    
    So we assume $S' = S$. Then $Y = G \intsec X$ is a subgroup of $X$, not the whole of $X$ since $G \neq X \rtimes S$. The conjugation action of $X \rtimes S$ on $X$ is the $S$-action. Since $G$ surjects onto $S$, the fact that $G$ is closed under conjugation by itself implies that $Y$ is closed under the $S$-action. Suppose that $Y$ is not maximal as an $S$-invariant subspace, $Y < Y' < X$. Let 
    \[
    GY' = \{gy : g \in G, y \in Y'\} \subset X \rtimes S.
    \]
    We claim that $GY'$ is a subgroup. Since $G$ and $Y'$ are subgroups, it suffices to show that any product $yg$, $y \in Y'$, $g \in G$, belongs to $GY'$. Write $g = sx$, $x \in X$, $s \in S$. Since $X$ is abelian,
    \[
    yg = ysx = sy^s x = sx y^s = g y^s,
    \]
    where $y^s$ denotes the conjugate $s^{-1} y s$.
    Since $Y'$ is $S$-invariant, the last product belongs to $GY'$. So $GY'$ is a subgroup.
    
    It is evident that $GY' \intsec X = Y'$. So $G < GY' < X \rtimes S$, contradicting maximality of $G$. So $Y$ is maximal, and $G$ is a group of type \ref{it:Y}.
  \end{proof}
  
  By Bhargava's result \cite[Theorem 1]{Bhargava_vdW}, the extension $\widetilde{K}_g$ has full Galois group $S_n$ for all but $O(H^{n})$ polynomials $g$ (or $O(H^{n-1})$ in the monic case), so Theorems \ref{thm:main} and \ref{thm:main_nonmonic} are already known for subgroups of type \ref{it:fibers}. To classify the subgroups of type \ref{it:Y}, we must find the maximal $S_n$-invariant subspaces of $X = \FF_2^n$, a vector space equipped with an $S_n$-action permuting the $n$ basis vectors freely. For convenience we let
  \[
  \0 = (0, \ldots, 0), \quad \1 = (1, \ldots, 1) \in X.
  \]
  We use a superscript $\perp$ to denote the orthogonal complement of a space under the $S_n$-invariant inner product
  \[
  (x_1,\ldots, x_n) \bullet (y_1,\ldots, y_n) = \sum_{i = 1}^n x_i y_i.
  \]
  \begin{lemma}\label{lem:invar_subsps}
    Let $X = \FF_2^n$, with the permutation action of $S_n$.
    The $S_n$-invariant subspaces of $X$ are as follows:
    \begin{itemize}
      \item $X$
      \item $\<\1\>^\perp = \{\mathbf{v} = (v_1,\ldots, v_n) : \sum_i v_i = 0\}$
      \item $\<\1\> = \{\0, \1\}$
      \item $\{\mathbf{0}\}$.
    \end{itemize}
  \end{lemma}
  \begin{proof}
    Let $W \subseteq X$ be an invariant subspace. If $W$ contains a vector $\mathbf{v} = (v_1,\ldots,v_n)$ besides $\0$ and $\1$, then upon applying some element of $S_n$ we can assume $v_1 = 1$, $v_2 = 0$. Then $W \ni \mathbf{v} + (12)\mathbf{v} = \<1, 1, 0, \ldots, 0\>$. Applying further permutations from $S_n$, we get that $W$ contains every vector with exactly two nonzero coordinates. Then $W$ contains their span, which is $\<\1\>^\perp$. Hence $W = \<\1\>^\perp$ or $W = X$.
  \end{proof}
  Among these, the maximal subgroups are
  \begin{itemize}
    \item $\<\1\>^\perp$, and
    \item $\<\1\>$, for $n$ odd (since for $n$ even, $\<\1\> \subseteq \<\1\>^\perp$).
  \end{itemize}
  Note that for $n$ odd, we have a direct sum decomposition $X = \<\1\> \oplus \<\1\>^\perp$. We now classify the groups $G$ using group cohomology as follows:
  \begin{lemma}\label{lem:H1 meaning}
    Let $X \rtimes S$ be a semidirect product of groups, with $X$ abelian, and let $K \subseteq X$ be a subgroup fixed by $S$. The subgroups of $G \subseteq X \rtimes S$ such that $G \intsec X = K$ are parametrized by $1$-cocycles $\epsilon \in Z^1(S, X/K)$, in other words, maps $\epsilon \colon S \to X/K$ satisfying the cocycle condition
    \begin{equation} \label{eq:coho_eps}
      \epsilon(\sigma \tau) = \epsilon(\sigma) + \sigma(\epsilon(\tau)).
    \end{equation}
    The map sends each $\epsilon$ to the group
    \begin{equation}\label{eq:coho_G}
      G = \{(x, \sigma) : x \equiv \epsilon(\sigma) \mod K\}.
    \end{equation}
    Moreover, two such subgroups $G, G'$ are conjugate if and only if the corresponding $\epsilon, \epsilon'$ are in the same cohomology class in $H^1(S, X/K)$; in other words, if there is a $y \in X$ such that for all $\sigma \in S$,
    \begin{equation}\label{eq:cobdry}
      \epsilon'(\sigma) = \epsilon(\sigma) + \sigma(y) - y.
    \end{equation}
  \end{lemma}
  \begin{proof}
    This is a fairly standard use of group cohomology. Indeed, it is easy to check that $G$ must have the form \eqref{eq:coho_G}, that closure under multiplication enforces \eqref{eq:coho_eps}, and that conjugation by $y \in X$ induces \eqref{eq:cobdry}. (Since $G$ surjects onto $S$, conjugation by $X$ is sufficient to produce all the conjugates of $G$ in $X \rtimes S$.)
  \end{proof}
  
  \begin{theorem} \label{thm:G_i}
    The maximal subgroups of $S_2 \wr S_n$ whose projection onto $S_n$ is the whole group are, up to conjugation, as follows:
    \begin{alignat*}{2}
      G_1 &= \left\{(\mathbf{v}, \sigma) \in \FF_2^n \rtimes S_n : \sum_i v_i = 0\right\} = \<\1\>^\perp \rtimes S_n, &\quad n &\geq 1 \\
      G_2 &= \left\{(\mathbf{v}, \sigma) \in \FF_2^n \rtimes S_n : \sum_i v_i = \sgn \sigma\right\}, &\quad n &\geq 2 \\
      G_3 &= \<\1\> \cross S_n, &\quad n &\geq 3 \text{ odd}
    \end{alignat*}
  \end{theorem}

  \begin{proof}
    By the preceding lemmas, we are left with computing $H^1(S_n, X/Y)$ for each of the maximal $S_n$-invariant subspaces $Y$ in Lemma \ref{lem:invar_subsps}.
    
    If $Y = \<\1\>^\perp$, then
    \[
    H^1(S_n, X/\<\1\>^\perp)
    = H^1(S_n, C_2)
    = \Hom(S_n, C_2) = C_2,
    \]
    the two maps $\epsilon$ being the zero map and the sign map, giving the subgroups $G_1$ and $G_2$ claimed.
    
    If $Y = \<\1\>$ for $n$ odd, we must compute
    \[
    H^1(S_n, X/\<\1\>).
    \]
    Since $n$ is odd, we have that $X = \<\1\> \oplus \<\1\>^\perp$ is a direct sum and $X/\<\1\> \isom \<\1\>^\perp$. First consider
    \[
    H^1(S_n, X).
    \]
    Note that with this action, $X = \Ind_{S_{n-1}}^{S_n} C_2$ is an induced module, so by Shapiro's lemma,
    \begin{align*}
      H^1(S_n, X) &= H^1(S_{n-1}, C_2) \\
      &= \Hom(S_{n-1}, C_2) \\
      &= C_2.
    \end{align*}
    Hence
    \[
    H^1(S_n, X/\<\1\>) = \frac{H^1(S_n, X)}{H^1(S_n, C_2)} = 0.
    \]
    Therefore there is only the trivial extension $G_3$.
    
    The restrictions on $n$ are provided due to the fact that, for some $n$, the $G_i$ are nonmaximal or coincident:
    \begin{itemize}
      \item For $n = 1$, $G_2 = G_1$ and $G_3 = S_2 \wr S_n$;
      \item For $n$ even, $G_3 \subseteq G_1$. \qedhere
    \end{itemize}
  \end{proof}
  
  \begin{rem}
    A proof of Theorem \ref{thm:G_i} without group cohomology is possible, but it involves some in-depth case analysis with many computations of products in $S_n$ and $S_2 \wr S_n$.  After a preprint form of this paper appeared, we learned of the useful reference \cite{BF} where these computations are carried out.
  \end{rem}
  
  \begin{rem}
    With a few more cohomological computations, we can classify \emph{all} the subgroups of $S_2 \wr S_n$ that surject onto $S_n$. They are the whole $S_2 \wr S_n$, the groups $G_1$, $G_2$, and $G_3$ (without parity restrictions on $n$), and the following additional groups:
    \begin{itemize}
      \item $\{0\} \cross S_n \isom S_n$
      \item $\left\{\bigl((\sgn \sigma)\1, \sigma\bigr) : \sigma \in S_n\right\}$, a twisted copy of $S_n$
      \item and, for $n = 4$ only, the group $\GL_2 \FF_3$, which acts on the $8$-element set $\FF_3^2 \setminus \{0\}$, preserving the partition into opposite pairs of vectors. In terms of its map to $S_4$, this is the double cover denoted $2\cdot S_4^+$ in \cite{WikiCover,GPCovers}. For instance, this is the generic Galois group of the even octic minimal polynomial of the $y$-coordinate of a $3$-torsion point on an elliptic curve over $\QQ$. It is a proper subgroup of the maximal subgroup $G_2$.
    \end{itemize}
    
  \end{rem}
  
  To prove Theorems \ref{thm:main} and \ref{thm:main_nonmonic}, we must bound the number of polynomials $f$, equivalently $g$, for which $G_f$ is (conjugate to) a subgroup of $G_1$, $G_2$, or $G_3$ for the values of $n$ listed in Theorem \ref{thm:G_i}.
  
  \begin{defn} \label{def:E_n}
    For $G \subseteq S_2 \wr S_n$ and $H \geq 2$, let $\E_n(G; H)$ be the number of separable reciprocal polynomials $f$ of degree $2n$ with coefficients in $[-H, H]$ such that $G_g = S_n$ and $G_f \subseteq G$. Let $\E_n^{\monic}(G; H)$ count the subset of these that are monic.
  \end{defn}
  
  The following lemma reduces computing $\E_n(G_1; H)$ and $\E_n(G_2; H)$ to number-theoretic conditions on $g$ that must hold in order to fulfill these conditions. (For $G_3$, we will use a different technique: see Section \ref{sec:G3}.)
  \begin{lemma}\label{lem:G123_conds} Assume that $G_g$ is the whole of $S_n$. Then:
    \begin{enumerate}[$($a$)$]
      \item\label{it:G1} $G_f \subseteq G_1$ if and only if $g(2) g(-2)$ is a square.
      \item\label{it:G2} $G_f \subseteq G_2$ if and only if $g(2) g(-2) \disc g$ is a square.
    \end{enumerate}
  \end{lemma}
  \begin{proof}
    For \ref{it:G1}, note that $G_1$ is the preimage of $A_{2n}$ under the usual inclusion $S_2 \wr S_n \hookrightarrow S_{2n}$. Thus $G_f \subseteq G_1$ if and only if $\disc f = g(2) g(-2) (\disc g)^2$ is a square, which happens exactly when $g(2) g(-2)$ is a square.
    
    For \ref{it:G2}, consider the embedding of $S_2 \wr S_n$ into $S_{3n}$ given by its action on the disjoint union of the roots of $f$ and of $g$. Note that $G_2$ is the preimage of $A_{3n}$ under this embedding. Hence $G_f \subseteq G_2$ if and only if $\disc f \cdot \disc g$ is a square, which happens exactly when $g(2) g(-2) \disc g$ is a square.
  \end{proof}
  
  \section{Counting \texorpdfstring{$G_1$}{G1}-polynomials}\label{sec:G1}
  We first deal with the case $G_1$, which yields the main term of Theorems \ref{thm:main} and \ref{thm:main_nonmonic}.
  \begin{theorem} \label{thm:G1}
    For $n \geq 1$,
    \begin{align}
      \E_n(G_1; H) &\asymp H^{n} \log H  \label{eq:G1_nonmonic}
      \intertext{and for $n \geq 2$,}
      \E_n^\monic(G_1; H) &\asymp H^{n - 1} \log H. \label{eq:G1_monic}
    \end{align}
  \end{theorem}
  
  By Lemma \ref{lem:G123_conds}\ref{it:G1}, it suffices to count $g$ such that $g(2) g(-2)$ is a square $z^2$.
  \begin{lemma} \label{lem:prod_square}
    The number of solutions to the equation $x y = z^2$, $1 \leq x, y, z \leq H$ ($H \geq 2$) is $\Theta(H \log H)$.
  \end{lemma}
  \begin{proof}
    A parametrization of the solutions is given by
    \[
    x = k u^2, \quad y = k v^2, \quad z = k u v
    \]
    where $k$, $u$, $v$ are positive integers and $\gcd(u,v) = 1$. For each $k$, $1 \leq k \leq H$, the pair $(u,v)$ is chosen from the box $1 \leq u, v \leq \sqrt{H/k}$, and the number of coprime pairs in this box is $\Theta(H/k)$ (the lower bound comes from citing the limiting proportion $6/\pi^2 > 0$ of coprime pairs when $H/k$ is large, and noting that there is always at least one solution $u = v = 1$). So the total number $N$ of solutions satisfies
    \begin{align*}
      N &\asymp \sum_{k = 1}^H \frac{H}{k} \\
      &\asymp H \log H,
    \end{align*}
    as desired.
  \end{proof}
  
  \begin{proof}[Proof of Theorem \ref{thm:G1}]
    For simplicity, we prove the non-monic case \eqref{eq:G1_nonmonic}.
    As $g(2), g(-2) \ll H$ and each pair $(x, y) = \(\size{g(2)}, \size{g(-2)}\)$ appears $O(H^{n-1})$ times, we get that $G_f \subseteq G_1$ at most $O(H^n \log H)$ times. Conversely, if we take $|x|, |y| \leq c H$ for an appropriate constant $c$, and with $x \equiv y \mod 4$ (which can be arranged, for instance by taking $4 | k$), we find that there are $\Theta(H^{n-1})$ polynomials $g$ with $g(2) = x$ and $g(-2) = y$, and thus $\Theta(H^{n} \log H)$ polynomials overall with Galois group $G_f \subseteq G_1$.
  \end{proof}
  
  \subsection{Remarks on the monic case}
  
  For the monic case, the argument is identical, replacing $n$ by $n - 1$. It is only necessary to have at least two free coefficients so that $g(2)$ and $g(-2)$ can be adjusted independently, requiring $n \geq 2$.

  \section{Counting \texorpdfstring{$G_2$}{G2}-polynomials}\label{sec:G2}
  
  For $G_2$, we prove the following bounds, which are stronger than those for $G_1$ by a factor of $\log H$:
  \begin{theorem} \label{thm:G2}
    For $n \geq 2$,
    \begin{equation}
      \E_n(G_2; H) \ll H^{n}.  \label{eq:G2_nonmonic}
    \end{equation}
    For $n \geq 3$,
    \begin{equation}
      \E_n^\monic(G_2; H) \ll H^{n - 1}. \label{eq:G2_monic}
    \end{equation}
  \end{theorem}
  
  By Lemma \ref{lem:G123_conds}\ref{it:G2}, we wish to count $g$ such that $g(2) g(-2) \disc g$ is a square. We use a sieve method adapted from Bhargava \cite{Bhargava_vdW}. We begin with some analytic preliminaries.
  
  \subsection{Twisted Poisson summation}
  Let $\Phi \colon \RR^n \to \CC$ be a Schwartz function. We normalize the Fourier transform by
  \[
  \widehat{\Phi}(y) = \int_{\RR^n} e^{-2\pi \sqrt{-1} x \bullet y} \Phi(x)\, dx.
  \]
  The usual Poisson summation formula
  \[
  \sum_{x \in \ZZ^n} \Phi(x) = \sum_{y \in \ZZ^n} \widehat\Phi(y)
  \]
  can be extended in various ways. If $L \subseteq \ZZ^n$ is a lattice (a subgroup of finite index), then
  \begin{equation} \label{eq:poisson}
    \sum_{x \in L} \Phi(x) = \frac{1}{[\ZZ^n : L]}\sum_{y \in L^*} \widehat\Phi(y),
  \end{equation}
  where $L^* \supseteq \ZZ^n$ is the dual lattice. More generally:
  \begin{prop}\label{prop:poisson}
    Let $L \subseteq \ZZ^n$ be a lattice, and let $\Psi \colon (\ZZ/M\ZZ)^n \to \CC$ be any function, where the modulus $M$ is coprime to $[\ZZ^n : L]$. Let $\widehat{\Psi} \colon (\ZZ/M\ZZ)^n \to \CC$ be the Fourier transform
    \[
    \widehat{\Psi}(y) = \frac{1}{M^n} \sum_{x \in (\ZZ/M\ZZ)^n} e^{-2\pi \sqrt{-1} x \bullet y/M} \Psi(x).
    \]
    For any Schwartz function $\Phi$,
    \[
    \sum_{x \in L} \Psi(x) \Phi(x) = \frac{1}{[\ZZ^n : L]}\sum_{y \in L^*} \widehat{\Psi}(y) \widehat{\Phi}\(\frac{y}{M}\).
    \]
  \end{prop}
  Observe that $\widehat{\Psi}$ is well defined on $L^*$ since $M$ is coprime to $[\ZZ^n : L]$. (In fact, one can do without this hypothesis, but then $\widehat{\Psi}$ becomes a function on $L^*/ML^*$ instead of being independent of $L$.) To prove this proposition, one may assume by linearity that $\Psi$ is supported on a single point, and the result reduces to Poisson summation.
  
  \subsection{The index of a polynomial over a field}
  
  Following Bhargava \cite[Proposition \red{26}]{Bhargava_vdW}, we make the following definition. If $P \in \Bbbk[x,y]$ is a nonzero homogeneous polynomial over a field, we factor $P = \prod_i P_i^{e_i}$ into powers of distinct irreducibles and define the \emph{index} of $P$ to be
  \[
  \ind(P) = \sum_i (e_i - 1) \deg P_i.
  \]
  The index is significant for bounding the power of a prime $p$ dividing the discriminant of the polynomial and the extension field it defines. The following lemma is used implicitly in \cite{Bhargava_vdW}; for completeness, we offer a statement and proof.
  \begin{lemma}\label{lem:index_tame}
    Let $g \in R[x,y]$ be a separable homogeneous binary form of degree $n$ over a Dedekind domain $R$. Let $F = \Frac R$, and let $E = F[\beta]/g(\beta, 1)$ be the \'etale algebra (product of separable finite field extensions) defined by $g$. If $\pp$ is a prime ideal of $R$ such that $E$ is tamely ramified at $\pp$ (e.g.\ $\ch(R/\pp) > n$) and $g$ is not identically zero modulo $\pp$, then
    \begin{equation}\label{eq:index_tame}
      v_\pp(\Disc E) \leq \ind(g \bmod \pp) \leq v_\pp(\disc g).
    \end{equation}
  \end{lemma}
  \begin{proof}
    We have deliberately stated the lemma in greater generality than needed to allow for making some reductions. First, we may replace $R$ by its completion at $\pp$. Let $\Bbbk$ be the residue field of $R$. Now, if $g \equiv g_1 g_2 \pmod \pp$ with $g_i \in \Bbbk[x,y]$ homogeneous and coprime, then by Hensel's lemma, the factorization lifts to $R$ and induces a splitting $E = E_1 \cross E_2$. The inequality \eqref{eq:index_tame} can then be deduced by summing the corresponding inequalities for $g_1$ and $g_2$. Thus we may assume that
    \[
    g \equiv c \cdot g_1^e \mod \pp,
    \]
    where $c$ is a constant and $g_1 \in \Bbbk[x,y]$ is irreducible of some degree $f$, with $e f = n$. We may change coordinates so that $g_1 \neq y$ is monic in $x$. Now
    \[
    E = \prod_{j = 1}^r E_j
    \]
    is a product of fields with the same inertia index $f$ and possibly different ramification indices $e_j$. We compute, using the usual formula for the discriminant of a tamely ramified extension:
    \begin{itemize}
      \item First, \begin{align*}
        v_\pp(\Disc E) &= \sum_{j = 1}^r v_\pp(\Disc E_j) \\
        &= \sum_{j = 1}^r (e_j - 1) f \\
        &= n - rf.
      \end{align*}
      \item $\ind(g \bmod \pp) = (e - 1) f = n - f$.
      \item Finally, we need to understand the ring $S = R[\beta]/g(\beta, 1) \subset E$. Note that $S$ is contained in
      \[
      S' = \{(x_1, \ldots, x_r) \in \OO_{E} : x_1 \equiv \cdots \equiv x_r \mod \pp\},
      \]
      a subring of $\OO_E$ of index $(r - 1)f$. Thus
      \begin{align*}
        v_\pp(\disc g) &= v_\pp(\Disc S) \\
        &\geq v_\pp(\Disc S') \\
        &= v_\pp(\Disc E) + 2 v_\pp([\OO_E : S']) \\
        &= n - rf + 2(r - 1) f \\
        &= n - f + (r - 1) f.
      \end{align*}
    \end{itemize}
    The desired inequality follows immediately. Equality for both parts holds exactly when $r = 1$, or in the original setup, when $\pp \nmid [\OO_E : S]$.
  \end{proof}
  
  Following Bhargava \cite[\textsection 5]{Bhargava_vdW}, we define
  \[
  D = \Disc K_g \textand C = \prod_{p\mid D} p
  \]
  and divide the counting into three cases based on the sizes of $C$ and $D$ relative to $H$. Unlike in \cite{Bhargava_vdW}, we do not have that $D$ is squarefull; but by Lemma \ref{lem:G123_conds}\ref{it:G2} we have that $D g(2) g(-2)$ is a square, which limits the cases.
  
  Let $\delta$ be a small constant (such as $1/4n$). For $R$ a ring, denote by $V^{\hom}(R)$ the $(n+1)$-dimensional space of binary $n$-ic forms $P(x,y)$ over $R$, and denote by $V(R)$ the space of polynomials $g(u)$ of degree at most $n$ over $R$. The two $R$-modules are isomorphic, but we will need to identify them in multiple ways.
  
  \subsection{Case I: \texorpdfstring{$C \leq H^{1+\delta}$}{C < H}, \texorpdfstring{$D \geq H^{2 + 2\delta}$}{D > H²}} \label{sec:case1}
  
  In this case, we need to estimate the number of $g$ for which $D = \Disc(K_g)$, $g(2)$, and $g(-2)$ have certain factors. We begin with a short argument that yields the result up to $\epsilon$.
  \begin{lemma}
    Let $p$ be a prime, and let $k$ be an integer. The number of binary forms $g \in V^{\hom}(\FF_p)$ such that
    \begin{itemize}
      \item $g(1,0) = 0$, that is, $y \mid g$, and
      \item $\ind(g) \geq k$
    \end{itemize}
    is $O(p^{n-k})$.
  \end{lemma}
  \begin{proof}
    We can immediately dispose of the case $p \leq n$, for here both the number of $g$ and the desired bound are $O(1)$. In \cite[Corollary \red{27}]{Bhargava_vdW}, it is shown that the number of degree-$n$ binary forms $g$ such that $\ind(g) \geq k$ is $O(p^{n+1-k})$. To impose the condition $g(1,0) = 0$, we can consider each $g$ as lying in the family of $p$ translates
    \[
    g(x,y + a x), \quad a \in \FF_p.
    \]
    The translates all have the same index, and if $p > n$, the translates are all distinct. Moreover, since $g$ has at most $n$ roots over $\FF_p$, at most $n = O(1)$ of the translates satisfy the added condition $g(1,0) = 0$. Hence the total number of such $g$ is $O(p^{n-k})$, as desired.
  \end{proof}
  
  Let $p > n$ be a prime dividing $C$, and suppose $p^k \parallel D$. By Lemma \ref{lem:index_tame} (left part), we have $\ind(g) \geq k$, and this occurs for a proportion $p^{-k}$ of $g$. If $k$ is odd, then we additionally have $p \mid g(2)$ or $p \mid g(-2)$, and altogether there is a proportion $p^{-k-1}$ of $g$ satisfying these conditions. Multiplying over $p$, the proportion of $G_2$-polynomials with $\Disc K_g = D$ is bounded by
  \[
  \prod_{p\mid D} O{\(p^{-2\ceil{k/2}}\)} = \frac{O\(c^{\omega(D_1)}\)}{D_1^2},
  \]
  where $D_1^2 = \prod_{p\mid D} p^{2\ceil{k/2}}$ is the least square divisible by $D$. Observe that $D_1 \gg H$ and that each $D_1$ occurs for at most $2^{\omega(D_1)}$ values of $D$. Moreover, these $g$ are cut out by congruence conditions mod $C$. If $C \leq H$, then we can estimate the number of lattice points very precisely because our modulus is lower than the size of the box. We get that the number of polynomials $g$ is
  \[
  \ll H^{n + 1} \sum_{D_1 \geq H} \frac{c^{\omega(D_1)}}{D_1^2} \ll_\epsilon H^{n + \epsilon}.
  \]
  
  Using Fourier analysis we can remove the $\epsilon$ and also extend the validity of this case from $C \leq H$ to $C \leq H^{1 + \delta}$.
  
  Recall some definitions from \cite[\textsection 4.1]{Bhargava_vdW}. If a binary $n$-ic form $f$ (over $\ZZ$, or over $\FF_p$) factors modulo $p$ as $\prod_{i=1}^r P_i^{e_i}$, with $P_i$ irreducible and $\deg(P_i)=f_i$, then the \emph{splitting type} $(f,p)$ of $f$ is defined as $(f_1^{e_1}\cdots f_r^{e_r})$, and the \emph{index} $\ind(f)$ of $f$ modulo~$p$ (or the \emph{index} of the splitting type $(f,p)$ of $f$) is defined to be $\sum_{i=1}^r (e_i-1)f_i$. If $p \mid f$, we put $\ind(f) = \infty$. More abstractly, a \emph{splitting type} is an expression $\sigma$ of the form $(f_1^{e_1}\cdots f_r^{e_r})$, where the $f_i$ and $e_i$ are positive integers. The \emph{degree} $\deg(\sigma)$ is $\sum_{i=1}^r e_i f_i$, and the \emph{index} $\ind(\sigma)$ is $\sum_{i=1}^r (e_i-1)f_i$. Finally, $\#{\Aut(\sigma)}$ is defined to be $\prod_i f_i$ times the number of permutations of the factors
  $f_i^{e_i}$ that preserve $\sigma$.
  
  In this work, we will need to deal with splitting types with a distinguished factor of degree $1$. If $\sigma$ has $f_1 = 1$, we let $\#{\Aut'(\sigma)}$ be $\prod_i f_i$ times the number of permutations of the factors $f_i^{e_i}$, $i \geq 2$, that preserve $\sigma$.
  
  We first recall the following lemma of Bhargava:
  \begin{lemma}[\cite{Bhargava_vdW}, Proposition \red{26}]\label{lem:ftgen}
    Let $\sigma=(f_1^{e_1}\cdots f_r^{e_r})$ 
    be a splitting type with $\deg(\sigma)=d\red{{} \leq n}$ and $\ind(\sigma)=k$.
    Let $w_{p,\sigma}:V^{\hom}(\FF_p)\to\CC$ be defined by
    \begin{equation*}
      w_{p,\sigma}(f) \coloneqq \parbox[t]{4.05in}{the number of $r$-tuples $(P_1,\ldots,P_r)$, up to the action of the group of permutations of $\{1,\ldots,r\}$ preserving $\sigma$, such that the $P_i$ are distinct irreducible binary forms where, for each $i$, we have $P_i(x,y)$ is $y$ or is monic as a polynomial in $x$, $\deg P_i=f_i$, and $P_1^{e_1}\cdots P_r^{e_r}\mid f$.}
    \end{equation*}
    Then 
    \begin{equation*}
      \widehat{w_{p,\sigma}}(g)=
      \begin{cases}
        {\displaystyle\frac{p^{-k}}{\#{\Aut}(\sigma)} + O(p^{-(k+1)})}	& \text{if $g=0$;}\\[.1in]
        {O(p^{-(k+1)})}		& \text{if $g \neq 0$}.\\
      \end{cases}
    \end{equation*}
  \end{lemma}
  
  The significance of this function $w_{p,\sigma}$ is twofold: First, $w_{p,\sigma}(f)$ is nonnegative, and is equal to $1$ when $f$ has splitting type $\sigma$; second, the definition is arranged so that $w_{p,\sigma}$ is a sum of characteristic functions of subspaces, which makes the Fourier transform nonnegative, small, and easily computable.
  
  Bhargava uses this lemma to bound the number of integer polynomials having high index at a set of primes. In our setting, we also need to bound the number of integer polynomials having high index \emph{and} which vanish at a given point mod $p$; hence we modify the lemma as follows:
  \begin{lemma}\label{lem:ftgen_pointed}
    Let $\sigma=(f_1^{e_1}\cdots f_r^{e_r})$ 
    be a splitting type with $f_1 = 1$. Let $\deg(\sigma)=d$ and $\ind(\sigma)=k$.
    Let $w'_{p,\sigma}:V^{\hom}(\FF_p)\to\CC$ be defined by
    \begin{equation*}
      w'_{p,\sigma}(f) \coloneqq \parbox[t]{4.05in}{the number of $r$-tuples $(P_2,\ldots,P_r)$, up to the action of the group of permutations of $\{2,\ldots,r\}$ preserving $\sigma$, such that the $P_i$ are distinct irreducible binary forms where, for each $i$, we have $P_i(x,y)$ is monic as a polynomial in $x$, $\deg P_i=f_i$, and $y^{e_1} P_2^{e_2} \cdots P_r^{e_r}\mid f$.}
    \end{equation*}
    Let $V^{\hom}_{e_1,\FF_p}$ denote the subspace of $V^{\hom}(\FF_p)$ comprising polynomials divisible by $y^{e_1}$, and let $\bigl(V^{\hom}_{e_1, \FF_p}\bigr)^\perp \subseteq V^{\hom}(\FF_p)^*$ be its dual. Then
    \begin{equation*}
      \widehat{w'_{p,\sigma}}(g)=
      \begin{cases}
        {\displaystyle\frac{p^{-(k+1)}}{\#{\Aut'}(\sigma)} + O(p^{-(k+2)})}	& \text{if $g \in \bigl(V^{\hom}_{e_1, \FF_p}\bigr)^\perp$;}\\[.1in]
        {O(p^{-(k+2)})}		& \text{if $g \notin \bigl(V^{\hom}_{e_1, \FF_p}\bigr)^\perp$}.\\
      \end{cases}
    \end{equation*}
  \end{lemma}
  \begin{proof}
    Rather than starting from scratch, we derive this lemma from the preceding one. Indeed, let $\sigma'$ be the splitting type obtained by deleting the first factor $f_1^{e_1} = 1^{e_1}$ from $\sigma$. Then $\Aut(\sigma') \isom \Aut'(\sigma)$, and $\ind(\sigma') = k - e_1 + 1$. We have that $w'_{p,\sigma}(f)$ vanishes unless $f \in V^{\hom}_{e_1,\FF_p}$ (implying in particular that $\widehat{w'}_{p,\sigma}$ is constant on cosets of $\bigl(V^{\hom}_{e_1,\FF_p}\bigr)^\perp$), and
    \[
    w'_{p,\sigma}(f) = \Breve{w}_{p,\sigma'}(f/y^{e_1})
    \]
    is \emph{almost} $w_{p,\sigma}(f/y^{e_1})$. We say ``almost'' because we need to exclude the case that one of the $P_i$ is $y$; so we define $\Breve{w}_{p,\sigma}$ to be just like $w_{p,\sigma}$, except that the $P_i(x,y)$ in the definition are not allowed to equal $y$. Since $w_{p,\sigma'}$ is a sum of characteristic functions of subspaces corresponding to the various choices of the $P_i$ and $\Breve{w}_{p,\sigma}$ is obtained by deleting some of the terms from this sum, we have, by Lemma \ref{lem:ftgen}, $\Breve{w}_{p,\sigma'}(f) \leq w_{p,\sigma'}(f)$ and
    \[
    \widehat{\Breve{w}_{p,\sigma'}}(g) \leq \widehat{w_{p,\sigma'}}(g) = \begin{cases}
      {\displaystyle\frac{p^{-(k-e_1+1)}}{\#{\Aut'}(\sigma)} + O(p^{-(k-e_1+2)})}	& \text{if $g \in V_{e_1, \FF_p}^\perp$;}\\[.1in]
      {O(p^{-(k-e_1+2)})}		& \text{if $g \notin V_{e_1, \FF_p}^\perp$}.
    \end{cases}
    \]
    Re-embedding $V^{\hom}_{e_1,\FF_p}$ into $V^{\hom}(\FF_p)$, the Fourier transform drops by a factor of $p^{e_1}$, yielding upper bounds of the claimed order of magnitude.
    
    To obtain that the main term is undiminished by the $w \mapsto \Breve{w}$ replacement, we can argue that
    \[
    \Breve{w}_{p,\sigma} = w_{p,\sigma} - \sum_{i : f_i = 1} \Breve{w}_{p,\sigma_i}
    \]
    where $\sigma_i$ is obtained by deleting $f_i^{e_i}$ from $\sigma$. Thus
    \begin{align*}
      \widehat{\Breve{w}_{p,\sigma}}(0) &= \widehat{w_{p,\sigma}}(0) - \sum_{i : f_i = 1} \widehat{\Breve{w}_{p,\sigma_i}}(0) \\
      &= \frac{p^{-(k-e_1+1)}}{\#{\Aut'}(\sigma)} + O(p^{-(k-e_1+2)}) + \sum_{i : f_i = 1} O(p^{-(k-e_1-e_i+1)}) \\
      &= \frac{p^{-(k-e_1+1)}}{\#{\Aut'}(\sigma)} + O(p^{-(k-e_1+2)}),
    \end{align*}
    as desired.
  \end{proof}

  We can now estimate the number of polynomials in Case I. 
  \begin{lemma} \label{lem:case1}
    Let $D$ be a positive integer. Let $C = \prod_{p\mid D} p$ be its radical, and let $D'^2 = \prod_{p\mid D} p^{2\ceil{v_p(D)/2}}$ be its smallest square multiple. Assume that $C < H^{1 + \delta}$. The number of reciprocal integer polynomials $f$ of height $\leq H$ for which $G_f \subseteq G_2$ and $D \mid \Disc K_g$ is
    \[
    \ll \frac{O(1)^{\omega(C)} H^{n+1}}{D'^2}.
    \]
  \end{lemma}
  \begin{rem}
    Note that in contrast to the notation in the rest of the paper, we do \emph{not} set $D = \Disc K_g$, but rather assume only that $D \mid \Disc K_g$. For Case I, we set $D = \Disc K_g$, but we will reuse this lemma in Case III, and there we will pick a general divisor $D$.    
  \end{rem}
  
  Note that Lemma \ref{lem:case1} implies the bound of Theorem \ref{thm:G2} on the number of $g$ in Case I because each $D'$ determines $D$ up to $O(1)^{\omega(C)}$ possibilities, and the total number of $g$ is thus
  \[
  \ll O(1)^{\omega(C)} H^{n+1} \sum_{D' \geq H^{1 + \delta}} \frac{1}{D'^2} \ll O(1)^{\omega(C)} H^{n-\delta} \ll H^{n - \delta + \epsilon} \ll H^n.
  \]
  
  \begin{proof}[Proof of Lemma \ref{lem:case1}]
    First of all, we divide out all primes $p \leq n$ from $D$. If there is at least one $K_g$ with $D \mid \disc K_g$, this change only affects $D$ by a bounded factor, since $v_p(\disc K_g)$ is uniformly bounded. Thus we can assume that every prime $p \mid C$ is at least $n$.
    
    For each prime $p \mid C$, \red{excluding the degenerate case that $g \equiv 0 \pmod p$ (see below), let $\sigma_p = (f_1^{e_1} \cdots f_r^{e_r})$ be the splitting type of its homogenization $\tilde g = y^n g(x/y)$.} By Lemma \ref{lem:index_tame} (left inequality),
    \red{\[
      v_p(D) \leq \ind(\tilde g \bmod p).
    \]
    Also, if $v_p(D)$ is odd, the relations
    \[
      1 \equiv v_p(D) \equiv v_p(\disc g) \equiv v_p\big(g(2) g(-2)\big) \mod 2
    \]
    imply that $g(u) \pmod p$ is divisible by either $u - 2$ or $u + 2$. Let $k_p = v_p(D') = \ceil{v_p(D)/2}$. We thus get one of three cases restricting $\sigma_p$ and $\tilde g$:
    \begin{enumerate}[$($a$)$]
      \item[$(0)$]\label{it:0} $\ind(\sigma_p) \geq 2 k_p$;
      \item[$(2)$]\label{it:2} $\ind(\sigma_p) \geq 2 k_p - 1$ and $f_1 = 1$ with $P_1 = x - 2y$;
      \item[$(-2)$]\label{it:-2} $\ind(\sigma_p) \geq 2 k_p - 1$ and $f_1 = 1$ with $P_1 = x + 2y$.
    \end{enumerate}}
    We call an \emph{annotated splitting type} a \red{pair $(\sigma,j)$ of a splitting type} $\sigma = \sigma_p$ with a choice of case \red{$j \in \{0, 2, -2\}$. Set
    \[
    w_{p,\sigma,j}(h) = \begin{cases}
      w_{p,\sigma}\big(y^n h(x/y)\big) & j = 0\\
      w'_{p,\sigma}\big(x^n h(y/x - j)\big) & j = \pm 2,
    \end{cases}
    \]
    the coordinate changes in the latter two cases being to send the distinguished linear factors $u - j$ in our situation to the distinguished linear factor $y$ in Lemma \ref{lem:ftgen_pointed}.
    Likewise, let
    \[
      \#{\Aut(\sigma,j)} = \begin{cases}
      \#{\Aut}(\sigma) & j = 0\\
      \#{\Aut'}(\sigma) & j = \pm 2.
    \end{cases}
    \]}
    
    For \red{$(\sigma_p, j_p)$} an annotated splitting type, let $\Psi_p = \red{w_{p,\sigma_p,j_p} \colon V_{\FF_p} \to \RR}$, \red{and let} $\Psi \colon V_{\ZZ/C\ZZ} \to \RR$ be the product of the $\Psi_p$. Observe that there are only $O(1)^{\omega(C)}$ choices of the annotated splitting types for all $p$, and for each $g$ as in the statement of the lemma, there is a choice for which $\Psi(g) \geq 1$. (In the degenerate case that \red{$g \equiv 0 \pmod p$}, we pick $\sigma_p$ \red{and $j_p$} arbitrarily, because $\red{w}_{p,\sigma\red{,j}}(0) \geq 1$ for all $\sigma$ \red{and $j$}.) Thus it suffices to prove, for a fixed choice of $\sigma_p$ and case (0), (2), or (-2) for each $p \mid C$, that
    \[
    \sum_{\Ht g \leq H} \Psi(g) \ll \frac{O(1)^{\omega(C)} H^{n+1}}{D'^2}.
    \]
    We claim that there is a decomposition $\Psi_p = \Lambda_p + \Delta_p$ with the following properties:
    \begin{itemize}
      \item $\Lambda_p = a_p \1_{L_p}$ is a rescaled characteristic function of a sublattice $L_p \subseteq V(\ZZ)$, with $\red{0 < {}} a_p \leq 1$, and $\widehat{\Lambda_p} = \hat a_p \1_{L_p^\perp}$ where $\red{0 < {}} \hat a_p \leq p^{-2 k_p}$;
      \item $\Delta_p$ has uniformly small Fourier transform: $\widehat{\Delta_p}(h) \ll p^{-2 k_p - 1}$.
    \end{itemize}
    Indeed, we take
    \[
    L_p = \begin{cases}
      V(\ZZ) & \red{j_p = 0} \\
      \red{\{g : g(j_p) \equiv g'(j_p) \equiv \cdots \equiv g^{(e_{1,p} - 1)}(j_p) \equiv 0 \mod p\}} & \red{j_p = \pm 2}.
    \end{cases}
    \]
    Note that in the \red{latter case}, $L_p$ is the image of $V^{\hom}_{e_1,\FF_p}$ under the \red{relevant identification} of $V^{\hom}$ with $V$ sending $x^i y^{n-i}$ to \red{$(u-j)^{n-i}$}. Choose the constant
    \[
    a_p = \frac{ [V(\ZZ) : L_p] \cdot p^{-{\ind \sigma} - \red{|j|/2}}}{\red{\#{\Aut(\sigma,j)}}} = \frac{p^{-{\sum_{i > \red{|j|/2}} (e_i - 1)f_i}}}{\red{\#{\Aut(\sigma,j)}}} \leq 1
    \]
    so that $\widehat{\Lambda_p}$ is the dominant term
    \[
    \hat a_p \1_{L_p^\perp}, \quad \hat a_p = \frac{p^{-{\ind \sigma} - \red{|j|/2}}}{\red{\#{\Aut(\sigma,j)}}} \leq p^{-2 k_p}
    \]
    of $\widehat{\Psi_p}$ as computed in Lemmas \ref{lem:ftgen} and \ref{lem:ftgen_pointed}, which show the smallness of $\widehat{\Delta_p}$.
    
    Now
    \[
    \Psi = \prod_p \(\Lambda_p + \Delta_p\) = \sum_{q \mid C} \Lambda_q \Delta_{C/q},
    \]
    where we set
    \[
    \Lambda_q = \prod_{p \mid q} \Lambda_p = a_q \1_{L_q}, \quad a_q = \prod_{p \mid q} a_p, \quad L_q = \bigintsec_{p \mid q} L_p, \textand \Delta_q = \prod_{p \mid q} \Delta_p.
    \]

    We now use Fourier analysis, as in Case I of \cite{Bhargava_vdW}, to rewrite the desired count of $g$. Let $\phi \colon \RR \to \RR$ be a Schwartz function on $\RR$ with the following properties:
    \begin{enumerate}[$($a$)$]
      \item $\phi(u)$ is nonnegative, and $\phi(u) \geq 1$ for $\size{u} \leq 1$;
      \item $\phi$ is compactly supported;
      \item The Fourier transform $\widehat{\phi}$ is real and nonnegative.
    \end{enumerate}
    Such a $\phi$ can be constructed, for instance, by taking a usual even ``bump function'' and convolving with itself, which squares the Fourier transform, ensuring nonnegativity. Then let $\Phi \colon V(\RR) \to \RR$ be the product $\Phi(g) = \phi(b_0)\phi(b_1) \cdots \phi(b_n)$. We use this as a smoothing factor:
    \begin{align*}
      X&\coloneqq \#{\{g \in V(\ZZ) : G_f \subseteq G_2, D \mid \disc K_g, \Ht g \leq H\}} \\
      &\leq \sum_{\Ht g \leq H} \Psi(g) \\
      &\leq \sum_{g \in V(\ZZ)} \Psi(g) \Phi\(\frac{g}{H}\) \\
      &= \sum_{q \mid C} \sum_{g \in V(\ZZ)} \Lambda_q(g) \Delta_{C/q}(g) \Phi\(\frac{g}{H}\).
    \end{align*}
    Since $C$ has $O(1)^{\omega(C)}$ divisors $q$, it suffices to prove that for all $q \mid C$,
    \begin{equation*}
      X_q \coloneqq \sum_{g \in V(\ZZ)} \Lambda_q(g) \Delta_{C/q}(g) \Phi\(\frac{g}{H}\) \ll \frac{O(1)^{\omega(C)} H^{n+1}}{D'^2}.
    \end{equation*}
    We apply twisted Poisson summation \ref{prop:poisson} with modulus $M = C/q$, which is coprime to $[V(\ZZ) : L_q] \mid q^n$, to get
    \begin{align}
      X_q &= a_q \sum_{g \in L_q} \Delta_{C/q}(g) \Phi\(\frac{g}{H}\) \nonumber \\
      &= \frac{a_q H^{n+1}}{[V(\ZZ) : L_q]} \sum_{h \in L_q^*} \widehat{\Delta_{C/q}} (h) \widehat\Phi\(\frac{q H h}{C}\) \label{z:poisson1} \\
      &\ll \frac{a_q H^{n+1}}{[V(\ZZ) : L_q]} \prod_{p \mid \frac{C}{q}} O\(p^{-2k_p - 1}\) \sum_{h \in L_q^*} \widehat{\Phi}\(\frac{q H h}{C}\). \nonumber
      \intertext{We apply Poisson summation again, now untwisted \eqref{eq:poisson}, to get}
      X_q &\ll \frac{a_qC^{n+1}}{q^{n+1}} \prod_{p \mid \frac{C}{q}} O\(p^{-2k_p - 1}\) \sum_{g \in L_q} \Phi\(\frac{C g}{q H}\) \nonumber \\
      &\leq \frac{a_qC^{n+1}}{q^{n+1}} \prod_{p \mid \frac{C}{q}} O\(p^{-2k_p - 1}\) \sum_{g \in V(\ZZ)} \Phi\(\frac{C g}{q H}\) \nonumber \\
      &\ll a_q \prod_{p \mid \frac{C}{q}} O\(p^{-2k_p - 1}\)\(\frac{C}{q} \max\left\{\frac{qH}{C}, 1\right\}\)^{n+1} \nonumber \\
      &\ll O(1)^{\omega(C)} \prod_{p \mid q} p^{-2k_p} \prod_{p \mid \frac{C}{q}} O\(p^{-2k_p - 1}\) \cdot \max\left\{H, \frac{C}{q}\right\}^{n+1} \nonumber \\
      &= \frac{O(1)^{\omega(C)} q}{CD'^2} \max\left\{H, \frac{C}{q}\right\}^{n+1}. \label{eq:final_max}
    \end{align}
    If the first argument to the maximum dominates, we get a bound
    \[
    X_q \ll \frac{O(1)^{\omega(C)}q}{C D'^2} H^{n+1} \leq \frac{O(1)^{\omega(C)} H^{n+1}}{D'^2},
    \]
    as desired. If instead the second argument dominates, we get a bound
    \[
    X_q \ll \frac{O(1)^{\omega(C)}q}{C D'^2} \(\frac{C}{q}\)^{n+1} = \frac{O(1)^{\omega(C)} C^n}{q^n D'^2} \leq \frac{O(1)^{\omega(C)} H^{n + n\delta}}{D'^2},
    \]
    as desired, since $n \delta \leq 1$.
  \end{proof}
  
  \subsection{Case II: \texorpdfstring{$D \leq H^{2 + 2\delta}$}{D < H²}} \label{sec:case2}
  In this case, for $n \geq 4$, we simply invoke Case II of Bhargava's treatment \cite[\textsection 5]{Bhargava_vdW}, which shows that the number of irreducible $g$ of height $< H$ defining a number field $K_g$ of primitive Galois group $G_g$ and discriminant $D \leq H^{2+2\delta}$ is $O(H^n)$ (or $O(H^{n-1})$ in the monic case). In our setting we are only concerned with the case $G_g = S_n$. We do not need to use the added knowledge that $G_f \subseteq G_2$.
  
  This gives the required bound for $n \geq 4$, but the method does not give a sufficient saving for $n \leq 3$. The low-degree cases needed here---the nonmonic cases $n=2,3$ and the monic case $n=3$---are proved separately in Appendix~\ref{app:low-degree-caseII}. The monic case $n=2$ remains open; see Remark \ref{rem:C4} above.
  
  \subsection{Case III: \texorpdfstring{$C \geq H^{1+\delta}$}{C > H}} \label{sec:case3}
  Here we adapt the method of Bhargava using the double discriminant.
  
  \begin{lemma}[\cite{Bhargava_vdW}, Proposition \red{34}]\label{lem:p_reasons}
    Let $p > 2$ be a prime. If $h(x_1, \ldots , x_n)$ is an integer polynomial, such that $h(c_1, \ldots , c_n)$ is a multiple of $p^2$ for mod $p$ reasons, that is, $h(c_1 + p d_1, \ldots , c_n + p d_n)$ is a multiple of $p^2$ for all $(d_1, \ldots , d_n) \in \ZZ^n$, then
    $\frac{\partial}{\partial x_n} h(c_1, \ldots , c_n)$ is a multiple of $p$.
  \end{lemma}
  
  Let
  \[
  h = g(2) g(-2) \disc g,
  \]
  considered as a polynomial in the coefficients $b_0, \ldots, b_n$ of $g$. Let $p \mid C$, $p > n$. We claim that $p^2 \mid h$ for mod $p$ reasons. We have that $p \mid D$, and $p^2$ divides the square $D g(2) g(-2)$. If $p^2 \mid D$, then by the left part of Lemma \ref{lem:index_tame}, the index of $g$ modulo $p$ is at least $2$. So by the right part of that same lemma, we have $p^2 \mid \disc g$ for mod $p$ reasons. Otherwise, we have $p \mid D$, so $p \mid \disc g$, and either $p \mid g(2)$ or $p \mid g(-2)$. Thus in all cases the product $h$ is divisible by $p^2$ for mod $p$ reasons. By Lemma \ref{lem:p_reasons}, this implies that the derivative $\frac{\partial}{\partial b_0} h$ with respect to the constant term is divisible by $p$. Hence their resultant
  \[
  R(b_1, \ldots, b_{n}) = \Res_{b_0} \(h, \frac{\partial}{\partial b_0} h\) = \pm b_n \disc_{b_0} h
  \]
  is a multiple of $p$.
  
  The polynomial $R$ is the analogue in our setting of the \emph{double discriminant} $\mathrm{DD}$ of \cite[Proposition \red{34}]{Bhargava_vdW}. For brevity, let $G = \disc_u g$. Then
  \begin{align}
    &R(b_1, \ldots, b_{n}) \nonumber \\
    &= \pm b_n \disc_{b_0} \bigl(g(2) g(-2) G\bigr) \nonumber \\
    &= \pm b_n \disc_{b_0} G \cdot \Res_{b_0}\bigl(g(2), g(-2)\bigr)^2 \Res_{b_0}\bigl(g(2), G\bigr)^2 \Res_{b_0}\bigl(g(-2), G\bigr)^2 \nonumber \\
    &= \pm b_n \disc_{b_0} G \cdot \bigl(g(2) - g(-2)\bigr)^2 \Bigl(\disc_u \bigl(g - g(2)\bigr)\Bigr)^2 \Bigl(\disc_u \bigl(g - g(-2)\bigr)\Bigr)^2, \label{eq:fzn_R}
  \end{align}
  where in the last step, we took advantage of the fact that $g(\pm 2)$ are linear in $b_0$ to use the standard formula
  \[
  \Res_x\bigl(x - a, P(x)\bigr) = P(a).
  \]
  In particular, $R$ is a product of which one factor is the double discriminant $\mathrm{DD}(g) = \disc_{b_0} G = \disc_{b_0} \disc_u g$. One easily sees from the factorization \eqref{eq:fzn_R} that $R$ is not identically zero as a function of $b_1,\ldots, b_{n}$.
  
  We now proceed as in \cite{Bhargava_vdW}. The number of $b_1, \ldots, b_{n} \in [-H, H]^{n}$ such that $R(b_1, \ldots, b_{n}) = 0$ is $O(H^{n-1})$ (by, e.g., \cite[Lemma 3.1]{Bhargava_geosieve}), and so the number of $g$ with such $b_1, \ldots, b_{n}$ is $O(H^{n})$.
  We now fix $b_1, \ldots, b_{n}$ such that $R(b_1, \ldots, b_{n}) \neq 0$. Then $R(b_1, \ldots, b_{n})$
  has at most $O_\epsilon(H^\epsilon)$ factors $C > H$. Once $C$ is determined by $b_1, \ldots, b_{n}$ (up to $O_\epsilon(H^\epsilon)$
  possibilities), then the number of solutions for $b_0$ (mod $C$) to $h \equiv 0 \mod C$ is
  $(\deg_{b_0} h)^{\omega(C)} = O_\epsilon(H^\epsilon)$, as the number of possibilities for $b_0$ modulo $p$ such that
  $h \equiv 0 \mod p$ for each $p \mid C$ is at most $\deg_{b_0}(h)$. Since $C > H$, the number of possibilities for $b_0 \in [-H, H]$ is also at most
  $O_\epsilon(H^\epsilon)$, and so the total number of $g$ in this case is $O_\epsilon(H^{n+\epsilon})$.
  
  To eliminate the $\epsilon$, we divide into two subcases, the first of which is reduced to Case I, just as in \cite{Bhargava_vdW}:
  
  \subsubsection*{Subcase (i):} $A=\displaystyle\prod_{\substack{\scriptstyle p\mid C \\ \scriptstyle p>H^{\delta/2}}} p \leq H$.
  
  In this subcase, $C$ has a factor $C_\I$ between $H^{1+\delta/2}$ and $H^{1+\delta}$, with $A\mid C_\I \mid C$. Pick $C_\I$ to be the largest such factor, and let $D_\I = \prod_{p\mid C_\I} p^{v_p(D)}$. We now appeal to Lemma \ref{lem:case1} from Case I, with $C_\I$ in place of $C$, and $D_\I$ in place of $D$. (Note that $D$ determines $C$, $C_\I$, and $D_\I$.) We get that the number of $g$ with a given $D$ is at most
  \[
  \ll \frac{O(1)^{\omega(C_\I) H^{n+1}}}{D_\I'^2},
  \]
  where
  \[
  D_\I' = \prod_{p\mid C_\I} p^{\ceil{v_p(D)/2}} \geq C_\I \geq H^{1 + \delta/2}.
  \]
  For each $D_\I'$, there are at most $2^{\omega(D)}$ values of $D_\I$, so the total number of reciprocal polynomials in this subcase is
  \[
  \sum_{D_\I' > H^{1 + \delta/2}} \frac{O(1)^{\omega(C_\I)} H^{n+1}}{D_\I'^2} = O_\epsilon(H^{n - \delta/2+\epsilon}) =O(H^{n}).
  \]
  \subsubsection*{Subcase (ii):} $A=\displaystyle\prod_{\substack{\scriptstyle p\mid C \\ \scriptstyle p>H^{\delta/2}}} p > H$.
  
  In this subcase, we carry out the original argument of Case III, with $C$ replaced by~$A$.
  We have $A\mid R(b_1,\ldots,b_{n})$. 
  
  Fix one of the $O(H^n)$ choices of $b_1,\ldots,b_{n}$ such that $R(b_1,\ldots,b_{n})\neq 0$.   Being bounded above by a fixed power of $H$, we see that 
  $R(b_1,\ldots,b_{n})$ can have at most a bounded number of possibilities for the factor $A$ (since all prime factors of $A$ are bounded below by a fixed positive power of $H$).
  Once $A$ is determined by $b_1,\ldots,b_{n}$,   then the number of solutions for $b_0$ (mod $A$) to $h \equiv 0$ (mod $A$) is $O(n^{\omega(A)})=O(1)$. 
  Since $A>H$, the total number  of $f$ in this subcase is also   $O(H^{n})$, completing the proof of Theorem \ref{thm:G2} in the non-monic case.
  
  \subsection{Remarks on the monic case} \label{sec:G2_monic}
  It would be impractical to replicate the full proof of Theorem \ref{thm:G2} in the monic case, most of which consists of changing $n$ to $n-1$ and correcting calculations appropriately. The following changes are worthy of note:
  \begin{itemize}
    \item Unlike in the non-monic case, the space $V(R)$ of monic polynomials over a ring does not have a natural origin. We must fix one, such as $g(u) = u^n$, in order to carry out the Fourier analysis.
    \item In Case I, we replace Lemma \ref{lem:ftgen} with the following lemma from Bhargava:
    \begin{lemma}[\cite{Bhargava_vdW}, Proposition \red{30}]\label{lem:ftgen2}
      Let $\sigma=(f_1^{e_1}\cdots f_r^{e_r})$ be a splitting type with $\deg(\sigma)=d \leq n$ and $\ind(\sigma)=k$.  Let $w_{p,\sigma} \colon V(\FF_p) \to \CC$ be defined by
      \begin{equation*}
        w_{p,\sigma}(f) \coloneqq \parbox[t]{3in}{the number of $r$-tuples $(P_1,\ldots,P_r)$, up to the action of the group of permutations of $\{1,\ldots,r\}$ preserving $\sigma$, such that the $P_i$ are distinct irreducible monic polynomials with $\deg P_i=f_i$ for each $i$ and $P_1^{e_1}\cdots P_r^{e_r}\mid f$.}
      \end{equation*}
      Then 
      \begin{equation*}
        \widehat{w_{p,\sigma}}(g)=
        \begin{cases}
          {\displaystyle\frac{p^{-k}}{\#{\Aut}(\sigma)} + O(p^{-(k+1)})}	& \text{if } {g=0};\\
          {O(p^{-(k+1)})}		& \text{if $g\neq 0$ and $d<n$;}\\
          O(p^{-(k+1/2)}) & \text{if $g\neq 0$ and $d=n$.} 
        \end{cases}
      \end{equation*}
    \end{lemma}
    
    \item We then replace Lemma \ref{lem:ftgen_pointed} as follows:
    
    \begin{lemma}
      Let $\sigma=(f_1^{e_1}\cdots f_r^{e_r})$ be a splitting type with $f_1 = 1$. Let $\deg(\sigma)=d$ and $\ind(\sigma)=k$.  Let $w'_{p,\sigma} \colon V(\FF_p) \to \CC$ be defined by
      \begin{equation*}
        w'_{p,\sigma}(f) \coloneqq \parbox[t]{3.2in}{the number of $r$-tuples $(P_1,\ldots,P_r)$, up to the action of the group of permutations of $\{2,\ldots,r\}$ preserving $\sigma$, such that the $P_i$ are distinct irreducible monic polynomials with $\deg P_i=f_i$ for each $i$, $f_1 = x$, and $P_1^{e_1}\cdots P_r^{e_r}\mid f$.}
      \end{equation*}
      Let $V_{e_1,\FF_p}$ denote the subspace of $V(\FF_p)$ comprising polynomials divisible by $x^{e_1}$, and let $V_{e_1, \FF_p}^\perp \subseteq V(\FF_p)^*$ be its dual. Then
      \begin{equation*}
        \widehat{w'_{p,\sigma}}(g)=
        \begin{cases}
          {\displaystyle\frac{p^{-(k+1)}}{\#{\Aut'}(\sigma)} + O(p^{-(k+2)})}	& \text{if $g \in V_{e_1, \FF_p}^\perp$}\\[.125in]
          {O(p^{-(k+2)})}		& \text{if $g \notin V_{e_1, \FF_p}^\perp$ and $d<n$;}\\[.02in]
          O(p^{-(k+3/2)}) & \text{if $g \notin V_{e_1, \FF_p}^\perp$ and $d=n$.} 
        \end{cases}
      \end{equation*}
    \end{lemma}
    \item Our $\Psi_p$ is then supported on
    \[
    \widetilde{L}_p = \begin{cases*}
      V(\ZZ) & case \ref{it:0} \\
      \{g : g(2) \equiv g'(2) \equiv \cdots \equiv g^{(e_{1,p} - 1)}(2) \equiv 0 \mod p\} & case \ref{it:2} \\
      \{g : g(-2) \equiv g'(-2) \equiv \cdots \equiv g^{(e_{1,p} - 1)}(-2) \equiv 0 \mod p\} & case \ref{it:-2},
    \end{cases*}
    \]
    which is no longer a lattice but a coset $g_p + L_p$ for some fixed monic polynomial $g_p$ and some lattice $L_p$ of index dividing $p^n$. The Fourier transform $\widehat{\Psi}$ is small away from $L_p^\perp$. The intersection
    \[
    \widetilde{L}_q = \bigintsec_{p \mid q} \widetilde{L}_p = g_q + L_q
    \]
    is likewise a coset of a lattice (if nonempty). When we carry out the twisted Poisson summation, the translation $g_q$ contributes a twist factor $e^{-2\pi \sqrt{-1}g_q \bullet h}$ to the values of $\widehat{\Lambda}_q$. Because we then immediately bound each term by its absolute value, this twist factor drops out.
    \item The final step of Case I, bounding $X_q$, proceeds as follows:
    \begin{align*}
      X_q &\ll \prod_{p \mid q} O\(p^{-2k_p}\) \prod_{p \mid \frac{C}{q}} O\(p^{-2k_p - 1/2}\) \cdot \max\left\{H, \frac{C}{q}\right\}^{n} \\
      &= \frac{O(1)^{\omega(C)} \sqrt{q}}{\sqrt{C} D'^2} \max\left\{H, \frac{C}{q}\right\}^{n} \\
      &= \frac{O(1)^{\omega(C)}}{D'^2} \max\left\{\frac{H^n\sqrt{q}}{\sqrt{C}}, \(\frac{C}{q}\)^{n-1/2}\right\} \\
      &\leq \frac{O(1)^{\omega(C)}}{D'^2} \max\left\{ H^n, H^{n - 1/2 + n\delta} \right\} \\
      &= \frac{O(1)^{\omega(C)} H^n}{D'^2}.
    \end{align*}
  \end{itemize}

  \section{Counting \texorpdfstring{$G_3$}{G3}-polynomials}\label{sec:G3}
  
  Finally, we count the polynomials $f$ having $G_f \subseteq G_3$ for each odd $n \geq 3$ (the even case is subsumed by $G_1$, as we noted in stating Theorem \ref{thm:G_i}).
  
  \begin{theorem} \label{thm:G3}
    For $n \geq 3$ odd,
    \begin{align}
      \E_n(G_3; H) &\ll \begin{cases}
        H^2 \log^2 H & n = 3 \\
        H^{\frac{n+1}{2}} & n \geq 5
      \end{cases}  \label{eq:G3_nonmonic} \\
      \E_n^\monic(G_3; H) &\ll \begin{cases}
        H^2 & n = 3 \\
        H^2 \log H \log \log H & n = 5 \\
        H^{\frac{n-1}{2}} \log H & n \geq 7.
      \end{cases}
      \label{eq:G3_monic}
    \end{align}
  \end{theorem}
  
  \subsection{Heights}
  We need a notion of height more general than the na\"ive height on integer polynomials used in Section \ref{sec:G2}.
  
  If $P = [x_0:\cdots: x_N] \in \PP^{N}(K)$ is a point in projective space over a number field $K$, we define its (exponential, projective, Weil) \emph{height} (denoted $H(P)$ in Silverman \cite[\textsection VIII.5]{SilvermanAEC}) as follows:
  \[
  \Htp P = \prod_{v} \max\big\{ \size{x_0}_v, \ldots, \size{x_N}_v \big\}^{[K_v : \QQ_v] / [K : \QQ]}
  \]
  where the product is taken over the places $v$ of $K$, and the norm $\size{x}_v$ extends the usual $v$-adic norm on $\QQ$. This normalization ensures that the height is unchanged if $K$ is embedded into a larger field. There are two natural ways to define a height on a polynomial $f(x) = a_n x^n + a_{n-1} x^{n-1} + \cdots + a_0$ over a number field, and we distinguish the \emph{projective} and \emph{affine} heights
  \begin{align}
    \Htp f &= \Htp\ [a_n : \cdots : a_0] \\
    \Ht f &= \Htp\ [a_n : \cdots : a_0 : 1].
  \end{align}
  For instance, if $f \in \ZZ[x]$ is nonzero, then
  $\Ht f = \max\{\size{a_0}, \ldots, \size{a_n}\}$ is the na\"ive height already introduced in Section \ref{sec:recip_polys}, while $\Htp f = \Ht f/\ct f$ is smaller by a factor of the \emph{content} $\ct(f) = \gcd(a_0, \ldots, a_n)$. In particular, we define heights of algebraic numbers $\alpha$ by $\Ht \alpha = \Htp\ [\alpha : 1]$. The following properties should be noted:
  \begin{itemize}
    \item If $\alpha$ is an algebraic integer with conjugates $\alpha = \alpha_1, \ldots, \alpha_n$, then
    \[
    \Ht \alpha = \prod_{i} \max\{1, |\alpha_i|\}^{1/n}
    \]
    is none other than the \emph{Mahler measure} of its minimal polynomial (suitably normalized).
    \item If $f$ has roots $\alpha_1, \ldots, \alpha_n$, then
    \begin{equation}\label{eq:height_poly_roots}
      2^{-n} \prod_{i} \Ht \alpha_i \leq \Htp f \leq 2^{n-1} \prod_i \Ht \alpha_i
    \end{equation}
    \cite[Theorem VIII.5.9]{SilvermanAEC}. In particular, for any factorization $f = g h$, we have
    \begin{equation}\label{eq:height_poly_fzn}
      \Htp f \asymp_{\deg f} \Htp g \cdot \Htp h.
    \end{equation}
  \end{itemize}

  \subsection{Proof of Theorem \ref{thm:G3}}
  \begin{proof}
    Under the conditions, there are only two possibilities for $G_f$: either $G_f \isom G_3 \cong S_2 \cross S_n$, or $G_f \cong S_n$. In the latter case, $f$ is reducible and factors as a product
    \[
    f(x) = c \cdot h(x) \cdot x^n h(1/x), \quad c \in \ZZ, \quad h \in \ZZ[x].
    \]
    By taking $|c| = \ct f$, we may assume that $\ct h = 1$. We have $\Htp f = \Ht f/|c| \leq H/|c|$, so by \eqref{eq:height_poly_fzn},
    \[
    \Ht h = \Htp h \ll \sqrt{\frac{H}{|c|}}.
    \]
    As $h$ has $n+1$ free coefficients, the number of polynomials $f$ we get is
    \begin{align*}
      &\ll \sum_{c = 1}^H \(\frac{H}{c}\)^{\frac{n+1}{2}} \ll H^{\frac{n+1}{2}}.
    \end{align*}
    
    We now assume that $G_f = G_3$. In this case there is a quadratic field $K_2 = \QQ(\sqrt{k})$, where $k$ is a squarefree integer, such that $K_f = K_g \cdot K_2$, so that over $K_2$, the polynomial $f$ factorizes:
    \begin{equation} \label{eq:fzn_G3}
      f(x) = c' \cdot h(x) \cdot x^n h(1/x), \quad c' \in K_2, \quad h \in K_2[x].
    \end{equation}
    We cannot assume that this factorization holds in $\OO_{K_2}[x]$ because $\OO_{K_2}$ need not be a UFD. However, we can rescale $h$ so that $h(1) \in \ZZ$. Then $c' = f(1) / h(1)^2 \in \QQ$, and the contents $c = \ct(f)$ and $\aa = \ct(h)$ satisfy
    \[
    c \OO_{K_2} = c' \cdot \aa^2.
    \]
    Therefore $\aa^2 = (c/c')$ is principal, generated by a rational number; in particular, $\aa = \ba{\aa}$ is self-conjugate. The self-conjugate ideals in a quadratic field are simply the rational multiples of products of ramified primes. Thus, after rescaling by a rational scalar, we have that there is a positive divisor $d \mid k$ such that
    \[
    \aa = \sqrt{(d)} = \begin{cases}
      \<d, \sqrt{k}\>, & k \equiv 2, 3 \mod 4 \\
      \ds \<d, \frac{k + \sqrt{k}}{2}\>, & k \equiv 1 \mod 4,
    \end{cases}
    \]
    and $c' = \pm c/d$. Note that $\ba{h}(x)$ and $x^n h(1/x)$ have the same roots, and comparing values at $x = 1$, they must be equal. Thus the coefficients $\theta_i \in K_2$ of $h(x) = \sum_{i = 0}^n \theta_i x^i$ satisfy
    \begin{equation} \label{x:refl_G3}
      \theta_i = \ba{\theta_{n-i}}.
    \end{equation}
    We now bound the height of the $\theta_i$. This requires us to control the \red{projective-to-affine} height ratio
    \[
    \red{\frac{\Htp h}{\Ht h}} = \prod_{v}
    \(\frac{\max\bigl\{\size{\theta_0}_v, \ldots, \size{\theta_n}_v\bigr\}}{\max\bigl\{ \size{\theta_0}_v, \ldots, \size{\theta_n}_v, 1\bigr\}}\)^{[(K_2)_v : \QQ_v] / [K_2 : \QQ]}.
    \]
    We consider the contributions of the different places $v$ in turn:
    \begin{itemize}
      \item If $v = p \mid d$ is finite, then the numerator is $\size{\aa}_p = 1/\sqrt{p}$, while the denominator is $1$. As $[(K_2)_v : \QQ_v] = [K_2 : \QQ] = 2$, we get a contribution $1/\sqrt{p}$ in this case.
      \item For all other finite $v$, the numerator and denominator are both $1$.
      \item If $v$ is infinite, the parenthesized fraction is $1$ because, by \eqref{x:refl_G3},
      \[
      \size{\theta_0}_v \cdot \size{\theta_n}_v = \size{\theta_0}_v \cdot \size{\ba{\theta_0}}_v = d \cdot \frac{\size{f(0)}}{c} \geq 1.
      \]
    \end{itemize}
    Hence
    \[
    \Ht h = \red{\sqrt{d} \Htp h \asymp \sqrt{d \Htp f} = \sqrt{\frac{d \Ht f}{c}}\leq \sqrt{\frac{d H}{c}}}.
    \]
    Write
    \[
    \theta_i = \frac{du_i + v_i \sqrt{k}}{2}, \quad u_i, v_i \in \ZZ
    \]
    and observe that not \emph{all} the $u_i$ nor \emph{all} the $v_i$ can be zero, or $f$ would be reducible. Now some coefficient $du_i$ of $h + \ba{h}$ is nonzero and thus at least $d$ in absolute value. So
    \[
    d \leq \Ht(h + \ba{h}) \ll \Ht(h) \ll \red{\sqrt{\frac{d H}{c}}},
    \]
    which gives us the bound 
    \[
    d \ll \frac{H}{c}.
    \]
    Analogously, analyzing the $v_i$ by considering $h - \ba{h}$ gives us the bound
    \begin{equation*}
      d' \coloneqq \frac{\size{k}}{d} \ll \frac{H}{c}.
    \end{equation*}
    Since $d u_i$ and $v_i \sqrt{\size{k}}$ are bounded by $\Ht h$, the numbers of possibilities for each $u_i$ and $v_i$ are $\ll \sqrt{H/c d}$ and $\ll \sqrt{H/c d'}$ respectively, and so the total number of polynomials is 
    \begin{align*}
      E_n(G_3; H) &\ll \sum_{c = 1}^H \sum_{d \ll \frac{H}{c}} \sum_{d' \ll \frac{H}{c}} \(\sqrt{\frac{H}{c d}} \sqrt{\frac{H}{c d'}}\)^{\frac{n+1}{2}} \\
      &= H^{\frac{n+1}{2}} \sum_{c = 1}^H c^{-\(\frac{n+1}{2}\)}\Biggl(\sum_{d \ll \frac{H}{c}} d^{-\(\frac{n+1}{4}\)} \Biggr)^2.
    \end{align*}
    If $n \geq 5$, the two remaining sums are both $O(1)$, and we get a bound of $O(H^{\frac{n+1}{2}})$. If $n = 3$, we get
    \begin{align*}
      E_n(G_3; H) &\ll H^{\frac{n+1}{2}} \sum_{c = 1}^H c^{-\(\frac{n+1}{2}\)} \log^2 H \red{{} \ll {}} H^{\frac{n+1}{2}} \log^2 H = H^2 \log^2 H. \qedhere
    \end{align*}
  \end{proof}
  
  \begin{rem}
    The ``$\ll$'' in Theorem \ref{thm:G3} can be sharpened to ``$\asymp$'' in the non-monic case by checking that a positive proportion of choices of the independent variables $c, d, d', u_i, v_i$ satisfy the needed conditions:
    \begin{itemize}
      \item $d$ and $d'$ are squarefree and coprime;
      \item $u_i$ and $v_i$ satisfy the conditions mod $2$ so that $\theta_i \in \aa$;
      \item and the $\theta_i$ do not all lie in an ideal strictly smaller than $\aa$.
    \end{itemize}
    All these can be managed using an appropriate form of the squarefree sieve. We omit the details.
  \end{rem}  
  
  \subsection{Remarks on the monic case}
  The monic case is proved similarly, but the possibilities are more restricted because we can factor $f(x) = \pm h_1(x) \cdot x^n h_1(1/x)$ over $\OO_{K_2}$, where $h_1(x)$ is monic. This scaling $h_1$ of $h$ may or may not be compatible with the one used above, but since $\ct(h_1) = (1)$, we derive that $\aa = (\theta)$ is principal. Since $\aa$ is known to be self-conjugate, we have a relation
  \begin{equation} \label{x:scp}
    \theta = \epsilon \ba{\theta}
  \end{equation}
  where $\epsilon \in \OO_{K_2}^\cross$ is a unit (necessarily of norm $1$). The relation \eqref{x:scp} determines $\theta$ up to scaling by $\QQ^\cross$; namely
  \[
  \theta = w(1 + \epsilon) \quad \text{and/or} \quad \theta = w\sqrt{k}(1 - \epsilon), \quad w \in \QQ^\cross,
  \]
  the label ``and/or'' being used because these formulas become invalid if $\epsilon = -1$, respectively $\epsilon = 1$. In particular, $\epsilon$ determines $d$. Moreover, since $\theta$ can be rescaled by a unit, it follows that rescaling $\epsilon$ by a square of a unit does not affect $d$. Since $\size{\OO_{K_2}^\cross / (\OO_{K_2}^\cross)^2}$ is uniformly bounded (at most $4$), there are $O(1)$ possible values of $d$ for each $k$. Each of the coefficients $\theta_1,\ldots,\theta_{(n-1)/2}$ has $O(H/\sqrt{k})$ values, as above. As to $\theta_0$, we have that $\eta = \theta_0^2 / d$ is a unit and determines $\theta_0$ up to sign. We have $\Ht \theta_0 \ll \sqrt{H/d}$ so $\Ht \eta \ll H$. Up to the finite group of roots of unity, $\eta$ is a power $\eta_1^m$ of the fundamental unit $\eta_1$ of $K_2$. We have $\size{\eta_1} \gg \sqrt{k}$, so the number of possibilities for $\eta$ is $\ll \log H/\log k$, for a grand total of
  \begin{equation*}
    \E_n^{\monic}(G_3; H) \ll \sum_{2 \leq k \ll H^2} \(\frac{H}{\sqrt{k}}\)^{\frac{n-1}{2}} \cdot \frac{\log H}{\log k}.
  \end{equation*}
  Estimating this sum yields the claimed bounds. In contrast to the non-monic case, it is unclear whether our bounds are sharp, specifically in the $n = 3$ and $n = 5$ cases. A more accurate count of monic $G_3$-polynomials demands a delicate understanding of the distribution of sizes of fundamental units in real quadratic fields. 
  
\appendix
\section{Low-degree cases in Case II}\label{app:low-degree-caseII}
The general Case II argument in Section~\ref{sec:case2} is insufficient  to give a $G_2$ estimate of lower order than the $G_1$ main term for $n=2$ and $3$. We supply here the estimates needed to complete Theorem~\ref{thm:G2}.

In contrast to the rest of this paper, the arguments used in the nonmonic and monic cases are rather different and will be presented separately.

\subsection{The nonmonic cases}

We first show that there are $O(H^n)$ polynomials $g$ of
height $O(H)$ for which
\[
  G_f \subseteq G_2 \textand C(K_g) < H^{1 + \delta},
\]
for $n = 2$ and $n = 3$.

We apply the Cayley-like transform
\begin{equation}\label{eq:F-def}
  F(X,Y)
  =
  (Y + X)^n
  g\left(2\cdot \frac{Y - X}{Y + X}\right),
\end{equation}
so that, writing
\[
  F(X,Y) = c_0 X^n + c_1 X^{n-1} Y + \cdots + c_n Y^n,
\]
we have
\begin{equation}\label{eq:endpoints}
  c_0 = g(-2),\qquad c_n = g(2).
\end{equation}
The transformation \eqref{eq:F-def} is linear in the coefficients of $g$, so $\Ht(F)\ll H$.  It also arises from a fractional linear change of
variable over $\QQ$, and hence $F$ defines the same degree $n$ number field $K_g$ as $g$. By Lemma~\ref{lem:G123_conds}\ref{it:G2}, in terms of $F$ the $G_2$-condition becomes
\begin{equation}\label{eq:G2_upon_Cayley}
   c_0 c_n \disc F\in\QQ^{\times2}.
\end{equation}

We need a uniform version of Lemma \ref{lem:prod_square}:

\begin{lemma}\label{lem:endpoints}
Uniformly for every squarefree integer $D$,
\[
  \#\bigl\{(x,y)\in\ZZ_{\neq 0}^2:
      |x|, |y| \leq H,\ D x y\in\QQ^{\times2}\bigr\}
  \ll \frac{\tau(|D|)}{\sqrt{|D|}} \cdot H\log H.
\]
\end{lemma}

\begin{proof}
Without loss of generality, we may restrict our attention to positive values of $D$, $x$, and $y$.
Put $k = \gcd(x,y)$ and let $x = k x'$, $y = k y'$ so that $\gcd(x', y') = 1$.
The condition that $D x y$ be a square is equivalent to $D x' y'$ being a
square.  Hence there is a factorization
$
  D = r s
$
such that
$
  x' = r u^2, y' = s v^2.
$
For fixed $r$ and $s$, there are at most
\begin{align*}
  \sum_{u \leq \sqrt{H/r}} \; \sum_{v \leq \sqrt{H/s}} \frac{H}{\max\{r u^2, s v^2\}}
\end{align*}
possibilities for $u$, $v$, and $k$. The sum splits into two parts according to which argument to the maximum dominates; by symmetry, we may assume $r u^2 \geq s v^2$, yielding
\begin{align*}
  \sum_{u \leq \sqrt{H/r}} \; \sum_{v \leq u \sqrt{r/s}} \frac{H}{r u^2} \ll \frac{H}{\sqrt{rs}}\sum_{u \leq \sqrt{H/r}} \frac{1}{u} \ll \frac{H \log H}{\sqrt{rs}} = \frac{H \log H}{\sqrt{D}}
\end{align*}
possibilities for fixed $r$ and $s$. Summing over factorizations of $D$ gives
\[
  \ll
  \frac{\tau(D)}{\sqrt D} \cdot H\log H
\]
solutions $(x, y)$, as claimed.
\end{proof}

\subsection{\texorpdfstring{$n = 2$}{n = 2} (reciprocal quartics)}

For $n = 2$, we simplify the notation to
\[
  F(X, Y) = a X^2 + b X Y + c Y^2,
\]
so the condition \eqref{eq:G2_upon_Cayley} becomes
\begin{equation}\label{eq:G2_quadratic}
  a c (b^2 - 4 a c) = z^2, \quad z \in \ZZ.
\end{equation}
Let $b^2 - 4 a c = D u^2$ with $D$ squarefree. (Up to an inconsequential factor of $4$, this agrees with the notation $D = \Disc K_g$ used above.) Note that $4 a c = D v^2$ for some $v \in \ZZ$. In particular, $D \mid 4 a c$, implying $D \mid b^2$ and (since $D$ is squarefree) $D \mid b$. In particular, $D \ll H$. For each of the $O(H)$ possible values of $D$, since $a c D$ is a square, there are at most
\[
  \frac{\tau(|D|)}{\sqrt{|D|}} \cdot H \log H
\]
possible pairs $(a,c)$ by Lemma \ref{lem:endpoints}. Finally, for fixed $a$ and $c$, the value of $b$ is constrained by the generalized Pell equation
\[
  b^2 - D u^2 = 4 a c,
\]
which has $O_\epsilon(H^\epsilon)$ solutions $(b, u)$, so overall the number of $G_2$-polynomials is
\[
  \ll_\epsilon H^{1 + \epsilon} \sum_{0 < D \ll H} \frac{\tau(D)}{\sqrt{D}} \ll_{\epsilon'} H^{3/2 + \epsilon'} = o(H^2),
\]
as desired.

\subsection{\texorpdfstring{$n = 3$}{n = 3} (reciprocal sextics)}

\begin{lemma}\label{lem:fixed-field-endpoints}
Fix a cubic field $K$ and nonzero integers $a,d$ with $|a|,|d|\ll H$.  For
every $\epsilon>0$, the number of integral binary cubic forms
\[
  F(X,Y)=aX^3+bX^2Y+cXY^2+dY^3,
  \qquad |b|,|c|\ll H,
\]
whose associated cubic algebra is $K$ is
\[
  O_\epsilon(H^\epsilon),
\]
uniformly in $K,a,d$.
\end{lemma}

\begin{proof}
Each such $F$ yields a generator $\omega$ for $K$ whose minimal polynomial is
\begin{equation}\label{eq:omega_min}
  0 = \frac{1}{a}F(\omega, a) = \omega^3 + b\omega^2 + a c\omega + a^2 d.
\end{equation}
Note that $\omega$ is an algebraic integer and that the ideal $(\omega)$ is a divisor of $(a^2 d)$. Therefore the number of possible ideals $(\omega)$ in $K$ is
\[
  O_\epsilon(|a^2 d|^\epsilon)
  = O_\epsilon(H^\epsilon).
\]

Once the ideal $(\omega)$ is fixed, its generator $\omega$ is known up to a unit in $\OO_K$. However, all complex roots of \eqref{eq:omega_min} have absolute value $O(H)$, so the logarithms of the units taking one allowable generator of $(\omega)$ to another lie in a region of diameter $O(\log H)$ in the logarithmic
unit lattice of $K$.  In fixed degree, the height of a non-torsion unit is
bounded uniformly away from zero.  Since the unit rank of a cubic field is at
most $2$, there are $O((\log H)^2) = O_\epsilon(H^\epsilon)$ allowable generators $\omega$ of any fixed ideal $(\omega)$.  Thus the total number of forms $F$ is $O_\epsilon(H^\epsilon)$ overall.
\end{proof}

We can now finish the count. Shankar and Thorne \cite[Theorem~2]{ShankarThorneGenInv} prove that there are
\begin{equation}\label{eq:ST}
  N_C(X)
  \coloneqq
  \#\{K:[K:\QQ]=3,\ C(K)\le X\}
  \asymp X\log X
\end{equation}
cubic fields of Wood conductor up to $X$. In our application, $X = H^{1 + \delta}$ so we get $O\left(H^{1 + \delta}\log H\right)$ fields $K$. For each $K$, let $\Disc K = D = D_0 u^2$ where $D_0$ is squarefree, and observe that the number of pairs $(a,d)$ satisfying the $G_2$-condition $a d D_0 \in \QQ^{\times2}$ is
\[
  \ll H \log H \frac{\tau(|D_0|)}{\sqrt{|D_0|}} \ll H \log H.
\]
Finally, for each $K,a,d$, Lemma~\ref{lem:fixed-field-endpoints} gives $O_\epsilon(H^\epsilon)$ forms.  Consequently
\[
\begin{aligned}
  \#\{F:C(K_F)<H^{1 + \delta},\ G_f\subseteq G_2\}
  &\ll_\epsilon
  \left(H^{1 + \delta}\log H\right)
  \left(H\log H\right)H^\epsilon\\
  &\ll_\epsilon
  H^{2 + \delta +\epsilon}(\log H)^2
  =o(H^3),
\end{aligned}
\]
as desired.

\subsection{The monic cubic case}

We next show that there are $O(H^2)$ monic cubic polynomials $g$ of
height $O(H)$ for which
\[
C(K_g) < H^{1 + \delta}.
\]
For this result we forget the condition $G_f\subseteq G_2$ and count all
monic cubics defining cubic fields. 

For a number field $K$ of degree $n$, let
\[
M_K(B)
=
\#\left\{
\begin{array}{c}
  g\in\ZZ[x]\text{ monic of degree }n:\ \Ht(g)\le B,\\
  \QQ[x]/(g)\isom K
\end{array}
\right\}.
\]
Lemke Oliver and Thorne \cite[Theorem~2.1]{LOTSmallGalois} prove
\begin{equation}\label{eq:LOT}
  M_K(B)
  \ll
  \frac{B(\log B)^{r_1+r_2-1}}{\lambda_K},
\end{equation}
where
\[
\lambda_K
=
\min_{\alpha\in\OO_K\setminus\ZZ}
\max_{\sigma:K\hookrightarrow\CC}|\sigma(\alpha)|.
\]
In our application, since $n = 3$ is prime, every $\alpha\in\OO_K\setminus\ZZ$ generates
$K$.  The Ellenberg--Venkatesh estimate recalled in
\cite[Lemma~3.1]{LOTSmallGalois} therefore yields
\[
\lambda_K\gg |{\Disc K}|^{1/{\ds(}n(n-1){\ds)}} = |{\Disc K}|^{1/6}.
\]
As $r_1+r_2-1\le2$, we obtain
\begin{equation}\label{eq:MK-cubic}
  M_K(B)
  \ll
  B(\log B)^2|{\Disc K}|^{-1/6}
  \le
  B(\log B)^2C(K)^{-1/6}.
\end{equation}

Partial summation, using the Shankar--Thorne estimate \eqref{eq:ST}, gives
\begin{align}
  \sum_{C(K)\le X}C(K)^{-1/6}
  &=
  X^{-1/6}N_C(X)
  +\frac16\int_1^X N_C(t)t^{-7/6}\,dt \nonumber\\
  &\ll X^{5/6}\log X.
  \label{eq:partial-summation}
\end{align}
Taking $B=O(H)$ and $X=H^{1 + \delta}$ in \eqref{eq:MK-cubic} and
\eqref{eq:partial-summation}, we find
\begin{align*}
  \#\{g:\Ht(g)\ll H,\ C(K_g)<H^{1 + \delta}\}
  &\ll
  H(\log H)^2
  \left(H^{1 + \delta}\right)^{5/6}\log H\\
  &=H^{11/6 + 5\delta/6}(\log H)^3.
\end{align*}
Since the exponent is less than $2$ for sufficiently small $\delta$ (our choice $\delta = 1/(4n) = 1/12$ suffices), this contribution is $o(H^2)$, as desired.

  \bibliography{ourbib_4,../Master}
  \bibliographystyle{plain}
\end{document}